\documentclass[12pt, reqno]{amsart}
\usepackage{}
\usepackage{cancel}
\usepackage{stmaryrd}
\usepackage{amsmath}
\usepackage{amsfonts}
\usepackage{amssymb}
\usepackage[all,cmtip]{xy}           
\usepackage{xspace}
\usepackage{bbding}
\usepackage{makecell}
\usepackage{txfonts}
\usepackage{enumerate}
\usepackage[shortlabels]{enumitem}
\usepackage{cite}
\usepackage{xcolor} 

\usepackage{tikz}
\usepackage{titletoc}

\usepackage{ifpdf}
\ifpdf
 \usepackage[colorlinks,final,hyperindex]{hyperref}
\else
 \usepackage[colorlinks,final,backref=page,hyperindex,hypertex]{hyperref}
\fi
\usepackage[active]{srcltx}

\makeatletter

\begin{document} 
\newcommand {\emptycomment}[1]{} 

\baselineskip=15pt
\newcommand{\nc}{\newcommand}
\newcommand{\delete}[1]{}
\nc{\mfootnote}[1]{\footnote{#1}} 
\nc{\todo}[1]{\tred{To do:} #1}

\nc{\mlabel}[1]{\label{#1}}  
\nc{\mcite}[1]{\cite{#1}}  
\nc{\mref}[1]{\ref{#1}}  
\nc{\meqref}[1]{\eqref{#1}} 
\nc{\mbibitem}[1]{\bibitem{#1}} 

\delete{
\nc{\mlabel}[1]{\label{#1}  
{\hfill \hspace{1cm}{\bf{{\ }\hfill(#1)}}}}
\nc{\mcite}[1]{\cite{#1}{{\bf{{\ }(#1)}}}}  
\nc{\mref}[1]{\ref{#1}{{\bf{{\ }(#1)}}}}  
\nc{\meqref}[1]{\eqref{#1}{{\bf{{\ }(#1)}}}} 
\nc{\mbibitem}[1]{\bibitem[\bf #1]{#1}} 
}

\newcommand {\comment}[1]{{\marginpar{*}\scriptsize\textbf{Comments:} #1}}
\nc{\mrm}[1]{{\rm #1}}
\nc{\id}{\mrm{id}}  \nc{\Id}{\mrm{Id}}
\nc{\admset}{\{\pm x\}\cup (-x+K^{\times}) \cup K^{\times} x^{-1}}

\def\a{\alpha}
\def\ad{associative D-}
\def\padm{$P$-admissible~}
\def\asi{ASI~}
\def\aybe{aYBe~}
\def\b{\beta}
\def\bd{\boxdot}
\def\bbf{\overline{f}}
\def\bF{\overline{F}}
\def\bbF{\overline{\overline{F}}}
\def\bbbf{\overline{\overline{f}}}
\def\bg{\overline{g}}
\def\bG{\overline{G}}
\def\bbG{\overline{\overline{G}}}
\def\bbg{\overline{\overline{g}}}
\def\bT{\overline{T}}
\def\bt{\overline{t}}
\def\bbT{\overline{\overline{T}}}
\def\bbt{\overline{\overline{t}}}
\def\bR{\overline{R}}
\def\br{\overline{r}}
\def\bbR{\overline{\overline{R}}}
\def\bbr{\overline{\overline{r}}}
\def\bu{\overline{u}}
\def\bU{\overline{U}}
\def\bbU{\overline{\overline{U}}}
\def\bbu{\overline{\overline{u}}}
\def\bw{\overline{w}}
\def\bW{\overline{W}}
\def\bbW{\overline{\overline{W}}}
\def\bbw{\overline{\overline{w}}}
\def\btl{\blacktriangleright}
\def\btr{\blacktriangleleft}
\def\calo{\mathcal{O}}
\def\ci{\circ}
\def\d{\delta}
\def\dd{\diamondsuit}
\def\D{\Delta}
\def\frakB{\mathfrak{B}}
\def\G{\Gamma}
\def\g{\gamma}
\def\gg{\mathfrak{g}}
\def\hh{\mathfrak{h}}
\def\k{\kappa}
\def\l{\lambda}
\def\ll{\mathfrak{L}}
\def\lh{\leftharpoonup}
\def\lr{\longrightarrow}
\def\N{Nijenhuis~}
\def\o{\otimes}
\def\om{\omega}
\def\opa{\cdot_{A}}
\def\opb{\cdot_{B}}
\def\p{\psi}
\def\r{\rho}
\def\ra{\rightarrow}
\def\rbs{Rota-Baxter system}
\def\rep{representation~}
\def\rh{\rightharpoonup}
\def\rr{\mathfrak{R}}
\def\s{\sigma}
\def\srbs{symmetric Rota-Baxter system}
\def\st{\star}
\def\ss{\mathfrak{sl}_2}
\def\ti{\times}
\def\tl{\triangleright}
\def\tr{\triangleleft}
\def\v{\varepsilon}
\def\vp{\varphi}
\def\vth{\vartheta}
\def\wn{\widetilde{N}}
\def\wb{\widetilde{\beta}}

\newtheorem{thm}{Theorem}[section]
\newtheorem{lem}[thm]{Lemma}
\newtheorem{cor}[thm]{Corollary}
\newtheorem{pro}[thm]{Proposition}
\theoremstyle{definition}
\newtheorem{defi}[thm]{Definition}
\newtheorem{ex}[thm]{Example}
\newtheorem{rmk}[thm]{Remark}
\newtheorem{tdef}[thm]{Theorem-Definition}
\newtheorem{condition}[thm]{Condition}
\newtheorem{question}[thm]{Question}
\renewcommand{\labelenumi}{{\rm(\alph{enumi})}}
\renewcommand{\theenumi}{\alph{enumi}}
\newcommand{\End}{\mathrm{End}}

\nc{\ts}[1]{\textcolor{purple}{MTS:#1}}
\nc{\zc}[2]{\textcolor{blue}{ZC:#2}}
\font\cyr=wncyr10


\title{Reynolds Leibniz bialgebras of any weight}

 \author[Ma]{Tianshui Ma \textsuperscript{*}}
 \address{School of Mathematics and Statistics, Henan Normal University, Xinxiang 453007, China}
         \email{matianshui@htu.edu.cn}

 \author[Ming]{Yuguang Ming}
 \address{School of Mathematics and Statistics, Henan Normal University, Xinxiang 453007, China}
         \email{mingyuguang@stu.htu.edu.cn}

 \author[Zhao]{Chan Zhao}
 \address{School of Mathematics and Statistics, Henan Normal University, Xinxiang 453007, China}
         \email{zhaochan2024@stu.htu.edu.cn}

\date{\today}

 \begin{abstract} This paper studies bialgebraic structures associated with a Reynolds Leibniz algebra of weight $\lambda$, that is, a Leibniz algebra equipped with a Reynolds operator of weight $\lambda$. We first present equivalent characterizations of Reynolds Leibniz bialgebras of weight $\lambda$, using matched pairs and Manin triples. Next, we examine compatibility conditions between solutions of the classical Leibniz Yang-Baxter equation and Reynolds operators of weight $\lambda$, framed in terms of triangular Reynolds Leibniz bialgebras. Finally, building on results of Ayupov {\em et al.}, we classify two-dimensional triangular Reynolds Leibniz bialgebras of weight $\lambda$.
 \end{abstract}

\subjclass[2020]{
17B38,  
17A30,  
16T25,   
16T10.   
}

\keywords{Weighted Reynolds operators; Leibniz bialgebras; classical Leibniz Yang-Baxter equation}

 \maketitle

 \vspace{-0.92cm}

  \tableofcontents

 \vspace{-1cm}

 \allowdisplaybreaks

\section{Introduction and preliminaries} This paper aims to develop a bialgebra theory for Reynolds Leibniz algebras of weight $\lambda$. We construct Reynolds Leibniz bialgebras explicitly by employing $\mathcal{O}$-operators and symmetric solutions of the classical Leibniz Yang-Baxter equation within this framework. In doing so, we establish compatibility conditions between the solutions of the classical Leibniz Yang-Baxter equation and Reynolds operators on Leibniz (co)algebras, which serves as the foundational step for our construction.

\subsection{Reynolds Leibniz algebras}
 The Reynolds operator originates from fluid dynamics, where it was introduced by O. Reynolds \cite{Rey} in the late 19th century to describe the averaging of turbulent flows-a technique now fundamental in the Reynolds-averaged Navier-Stokes equations. In recent years, the Reynolds operator has garnered considerable interest across various algebraic contexts, as evidenced by a growing body of literature, see \cite{Das, GGL, GD, LW24, WK, ZGG}, etc. Notably, in 2024, Guo and Das \cite{GD} studied the cohomology and deformation theory of a class of generalized Reynolds operators on Leibniz algebras. In the same year, Li and Wang \cite{LW24} introduced the notion of a weighted Reynolds Leibniz algebra and investigated the cohomology of such algebras with coefficients in an appropriate representation. Concretely, a {\bf Reynolds Leibniz algebra of weight $\l$} is a pair $((\gg, [,]), \rr)$ consisting of a Leibniz algebra $(\gg, [,])$, that is to say, a vector space $\gg$ and bilinear operation $[,]: \gg\o \gg\lr \gg$ satisfying, for all $x, y, z \in \gg$, 
 \begin{eqnarray}
 &[x, [y, z]]=[[x, y], z]+[y, [x, z]], &\label{eq:1}
 \end{eqnarray}
 and a linear map $\rr: \gg\to \gg$ satisfying, for all $x, y\in \gg$ and a scalar $\l\in K$, the equation below holds:
 \begin{eqnarray}
 &[\rr(x),\rr(y)]+\l \rr([\rr(x),\rr(y)])=\rr([x,\rr(y)])+\rr([\rr(x),y]).&\label{eq:2}
 \end{eqnarray}
 Also in 2024, Guo {\em et al.} \cite{GGL} proposed a further generalization of Reynolds algebras by replacing the scalar $\lambda$ in the above identity with an element of the underlying algebra itself, leading to a more flexible algebraic framework.

\subsection{Leibniz (bi)algebras}
 Leibniz algebras were originally introduced by Bloh \cite{B} under the name ``D-algebras" and were later systematically developed by Loday \cite{L, LP}. They serve as a natural non-commutative analogue of Lie algebras. The study of Leibniz algebras has attracted considerable attention in recent years, as evidenced by a growing body of work including \cite{DS,LMW,LW25,LW24,LS,LP,MSZ,MS25,MS24,MS22,TS,WK}, among others.

 In the context of Leibniz bialgebras, Rezaei-Aghdam {\em et al.} \cite{RSH} first introduced the concept by directly adapting the definition of a Lie bialgebra. Subsequently, Barreiro and Benayadi \cite{BB} proposed a novel approach, grounded in a fundamental result that identifies the underlying vector space of a symmetric Leibniz algebra as a Lie algebra (or a commutative associative algebra) via a naturally induced bracket (or product). Tang and Sheng \cite{TS} later investigated Leibniz bialgebras through the twisting theory of twilled Leibniz algebras and introduced the classical Leibniz Yang-Baxter equation (cLYBe). Following this, Li {\em et al.} \cite{LMW} recovered the cLYBe using Bai's method in \cite{Bai1} and demonstrated that a class of Leibniz bialgebras can give rise to Nijenhuis operators. More recently in 2025, Bai {\em et al.} \cite{BLST} explored factorizable quasitriangular Leibniz bialgebras, while Xu {\em et al.} \cite{XBH} examined Leibniz conformal bialgebras and the related classical Leibniz conformal Yang-Baxter equation. In parallel, Ma {\em et al.} \cite{MSZ1} studied Hom-deformations of Leibniz bialgebras.

\subsection{Main results and layout of the paper} This paper forges a theoretical link between the two fundamental aspects of Leibniz algebras discussed above: their bialgebraic structure and the theory of Reynolds operators. The key notions and constructions introduced in this work are summarized in the following diagram. 
$$ 
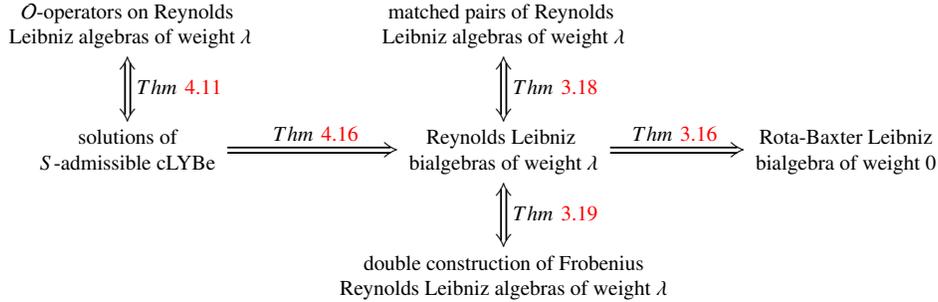
\xymatrix{
&
 \mathcal{O}\text{-operators on Reynolds}\atop\text{ Leibniz algebras of weight $\l$} \ar@2{<->}^{Thm~\ref{thm:126}}[d]
 &\text{matched pairs of Reynolds }\atop \text{Leibniz algebras of weight $\l$}
 &
 \\
&
\text{solutions of}\atop \text{$S$-admissible cLYBe}
\ar@2{->}^{Thm~\ref{thm:150}}[r]
& \text{Reynolds Leibniz }\atop \text{bialgebras of weight $\l$} \ar@2{->}^{Thm~\ref{thm:de:87}}[r] 
\ar@2{<->}_{Thm ~\ref{thm:91}}[u]
\ar@2{<->}^{Thm~\ref{thm:92}}[d]
&\text{Rota-Baxter Leibniz}\atop \text{bialgebra of weight 0} 
& \\
&
&\text{double construction of Frobenius}\atop \text{Reynolds Leibniz algebras of weight $\l$}
&
}
$$

 The paper is organized as follows. In Section \ref{se:rep}, we introduce the notions of a representation and a dual representation for a Reynolds Leibniz algebra. For completeness, we also discuss corepresentations of a Reynolds Leibniz coalgebra. Section \ref{se:3} presents the concept of a Reynolds Leibniz bialgebra and provides its equivalent characterizations, first via a matched pair and then through a Manin triple of Reynolds Leibniz algebras. Section \ref{se:4} focuses on a special class of these bialgebras, from which we derive compatibility conditions between classical Leibniz Yang-Baxter equations and Reynolds operators. The final section is devoted to the classification of two-dimensional triangular Reynolds Leibniz bialgebras.

\smallskip\smallskip
 \noindent{\bf Notations:} Throughout this paper, we fix a field $K$. All vector spaces, tensor products, and linear homomorphisms are over $K$. $\tau: M\o N\to N\o M$ is the flip map. In this paper, Reynolds Leibniz algebra means Reynolds Leibniz algebra of weight $\l$.

 \section{The (co)representation of Reynolds Leibniz (co)algebra} \label{se:rep}

 In this section, we introduce the notion of a (co)representation of a Reynolds Leibniz (co)algebra and discuss some related properties.

 \subsection{Representation of Reynolds Leibniz algebras}

 \begin{defi}\label{defi:315}\cite{LW24} Let $((\gg, [,]_\gg), \rr)$ and $((\gg^{'}, [,]_{\gg^{'}}), \rr^{'})$ be Reynolds Leibniz algebras. A {\bf homomorphism from $((\gg, [,]_\gg), \rr)$ to $((\gg^{'}, [,]_{\gg^{'}}), \rr^{'})$} consists of a Leibniz algebra homomorphism $\phi:\gg\rightarrow\gg^{'}$ such that
 \begin{eqnarray*}
 \phi\circ\rr=\rr^{'}\circ\phi.
 \end{eqnarray*}
 Furthermore, if $\phi$ is invertible, then $\phi$ is called an {\bf isomorphism}.
 \end{defi}

 \begin{pro}\label{pro:310} Let $((\gg,[,]),\rr)$ be a Reynolds Leibniz algebra. Then $\rr $ induce a new Leibniz algebra structure on $\gg$ given by
 \begin{eqnarray}
 &[x,y]_{\rr}:=[x,\rr(y)]+[\rr(x),y]-\lambda[\rr(x),\rr(y)], \quad \forall x,y\in\gg. &\label{eq:311}
 \end{eqnarray}
 The Leibniz algebra $(\gg,[,]_{\rr})$ is called the \textbf{induced Leibniz algebra}. Moreover, $((\gg,[,]_{\rr}),\rr)$ is also a Reynolds Leibniz algebra and $\rr$ is a homomorphism from $((\gg,[,]_{\rr}),\rr)$ to $((\gg,[,]),\rr)$.

 \begin{proof} For all $x, y, z\in\gg$, we calculate
 \begin{flalign*}
 &\hspace{-4mm}[x, [y, z]_{\rr}]_{\rr}-[[x, y]_{\rr}, z]_{\rr}-[y, [x, z]_{\rr}]_{\rr}\\
 &=[\rr(x),[y,\rr(z)]]-[[\rr(x),y],\rr(z)]-[y,[\rr(x),\rr(z)]]+[\rr(x),[\rr(y),z]]\\
 &\hspace{4mm}-[[\rr(x),\rr(y)],z]-[\rr(y),[\rr(x),z]]+[x,[\rr(y),\rr(z)]]-[[x,\rr (y)],\rr(z)]\\
 &\hspace{4mm}-[\rr(y),[x,\rr(z)]]-2\lambda\big([\rr(x),[\rr(y),\rr(z)]]-[[\rr(x),\rr (y)],\rr(z)]\\
 &\hspace{4mm}-[\rr(y),[\rr(x),\rr(z)]]\big)\\
 &\stackrel{(\ref{eq:1})}{=}0.
 \end{flalign*}
 Therefore $(\gg,[,]_{\rr})$ is a Leibniz algebra. Furthermore,
 \begin{flalign*}
 &[\rr(x),\rr(y)]_\rr+\lambda \rr([\rr(x),\rr(y)]_\rr)-\rr([x,\rr(y)]_\rr)-\rr([\rr (x),y]_\rr)\\
 &=[\rr(x),\rr(\rr(y))]+\lambda\rr([\rr(x),\rr(\rr(y))]-\rr([x,\rr(\rr(y))]-\rr([\rr (x),\rr(y)]+[\rr(\rr(x)),\rr(y)]\\
 &\hspace{4mm}+\lambda\rr([\rr(\rr(x)),\rr(y)]-\rr([\rr(x),\rr(y)])-\rr([\rr(\rr (x)),y])-\lambda\Big([\rr(\rr(x)),\rr(\rr(y))]\\
 &\hspace{4mm}+\lambda\rr([\rr(\rr(x)),\rr(\rr(y))]-\rr([\rr(x),\rr(\rr(y))]- \rr([\rr (\rr(x)),\rr(y)]\Big)\\
 &\stackrel{(\ref{eq:2})}{=}0.
 \end{flalign*}
 Then $((\gg,[,]_{\rr}),\rr)$ is a Reynolds Leibniz algebra and $\rr$ is homomorphism from $((\gg,[,]_{\rr}),\rr)$ to $((\gg,[,]),\rr)$ by $[\rr(x),\rr(y)]=\rr([x,y]_\rr)$.
 \end{proof}
\end{pro}

 \begin{defi}\label{defi:repLa}\cite{LW24,TS} A \textbf{representation of Leibniz algebra} $(\gg,[,])$ is a triple $(V, \rho{^L}, \rho{^R})$, where $V$ is a vector space , $\rho{^L}, \rho{^R}:\gg\lr \End(V)$ are linear maps, for all $x,y\in \gg$,
 \begin{eqnarray*}
 &\rho{^L}([x, y])=\rho{^L}(x)\circ\rho{^L}(y)-\rho{^L}(y)\circ\rho{^L}(x),&\\
 &\rho{^R}([x, y])=\rho{^L}(x)\circ\rho{^R}(y)-\rho{^R}(y)\circ\rho{^L}(x),& \\
 &-\rho{^R}(y)\circ\rho{^L}(x)=\rho{^R}(y)\circ\rho{^R}(x).&
 \end{eqnarray*}
 \hspace{4mm}Two representations $(V, \rho{^L}, \rho{^R})$ and $(V^\prime, \rho^{L^{\prime}}, \rho^{R^{\prime}})$ of a Leibniz algebra $(\gg,[,])$ are called equivalent if there is a linear isomorphism $f:V\longrightarrow V^\prime$ satisfying, for all $x\in \gg, v\in V$,
 \begin{eqnarray*}
 &f\Big(\rho^L(x)v\Big)=\rho^{L^{\prime}}(x)f(v), ~~~ f\Big(\rho^R(x)v\Big)=\rho^{R^{\prime}}(x)f(v).&
 \end{eqnarray*}
 \end{defi}

 \begin{rmk}\label{rmk:7}
 Let $(\gg,[,])$ be a Leibniz algebra, then $(\gg, L, R)$ is a representation (called the {\bf adjoint representation}) of $(\gg,[,])$, where $L_{x}(y)=[x,y]=R_{y}(x)$ are two linear maps, for all $x,y\in \gg$.
 \end{rmk}

 \begin{defi}\label{defi:8}\cite{LW24} Let $(\gg, \rr)$ be a Reynolds Leibniz algebra, $(V, \rho{^L}, \rho{^R})$ be a representation of Leibniz algebra $(\gg,[,])$, $\alpha:V\longrightarrow V$ be a linear map. If for all $x\in \gg$ and $v\in V$,
 \begin{eqnarray}
 &\rho^L(\rr(x))\alpha(v)+\lambda \a(\rho^L(\rr(x))\a(v))=\a(\rho^L(\rr (x))v)+\a(\rho^L(x)\a(v)),&\label{eq:3}\\
 &\rho^R(\rr(x))\alpha(v)+\lambda\a(\rho^R(\rr(x))\a(v))=\a(\rho^R(\rr (x))v)+\a(\rho^R(x)\a(v)).&\label{eq:4}
 \end{eqnarray}
 Then the quadruple $(V, \rho{^L}, \rho{^R}, \a)$ is a \textbf{representation of $(\gg, \rr)$}.

 Two representations $(V, \rho{^L}, \rho{^R}, \alpha)$ and $(V^\prime, \rho^{L^{\prime}}, \rho^{R^{\prime}}, {\alpha}^\prime)$ of $(\gg, \rr)$ are called {\bf equivalent} if $(V, \rho{^L}, \rho{^R})$ and $(V^\prime, \rho^{L^{\prime}}, \rho^{R^{\prime}})$ are equivalent for representations of $(\gg, [,])$, further, $f\circ \alpha={\alpha}^\prime\circ f$.
 \end{defi}

 \begin{rmk}\label{rmk:9}  Let $(\gg, \rr)$ be a Reynolds Leibniz algebra, $(V, \rho{^L}, \rho{^R})$ be a representation of Leibniz algebra $(\gg,[,])$. Then
 \begin{enumerate}[(1)]
   \item If $\l=0$, then a Reynolds Leibniz algebra of weight $\l$ is a Rota-Baxter Leibniz algebra of weight 0. In this case Definition \ref{defi:8} turns to be a representation of a Rota-Baxter Leibniz algebra of weight 0.
   \item \label{it:10} $(\gg, L, R, \rr)$ is a representation of $(\gg, \rr)$, called \textbf{the adjoint representation} of $(\gg, \rr)$.
 \end{enumerate}
 \end{rmk}

 Let $\rho{^L}, \rho{^R}: \gg\lr \End(V)$ be two linear maps. Define the multiplication on $\gg\ltimes V$ by:
 \begin{eqnarray*}
 &[x+u, y+v]_{\ltimes}=[x, y]+\rho^L(x)v+\rho^R(y)u, \quad \forall x,y\in \gg, u,v\in V.&\label{eq:301}
 \end{eqnarray*}
 Then $\gg\oplus V$ is a Leibniz algebra (denoted by $\gg\ltimes_{\rho{^L},\; \rho{^R}} V$, and called the \textbf{semi-direct product} of $(\gg,[,])$ via $(V, \rho{^L}, \rho{^R})$) if and only if $(V, \rho{^L}, \rho{^R})$ is a representation of Leibniz algebra $(\gg,[,])$ \cite{TS}.

 Similarly, the representation of the Reynolds Leibniz algebra can be characterized by the semi-direct product Reynolds Leibniz algebra, the proof is omitted here.

 \begin{pro}\label{pro:12} Let $(\gg, \rr)$ be a Reynolds Leibniz algebra, $(V, \rho{^L}, \rho{^R})$ be a representation of $(\gg,[,])$, $\alpha:V\longrightarrow V$ be a linear map. For all $x\in \gg$, $u\in V$, define a linear map:
 \begin{eqnarray}\label{eq:13}
 &\rr+\a: \gg\oplus V\lr \gg\oplus V, \quad  (\rr+\a)(x+u):=\rr(x)+\a(u).&\nonumber
 \end{eqnarray}
 Then $(\gg\oplus V, \rr+\a)$ is a Reynolds Leibniz algebra if and only if $(V, \rho{^L}, \rho{^R}, \a)$ is a representation of $(\gg, \rr)$. The resulting Reynolds Leibniz algebra is denoted by $(\gg\ltimes_{\rho{^L},\; \rho{^R}}V, \rr+\a)$ and is called the \textbf{semi-direct product} of $(\gg, \rr)$ by its representation $(V, \rho{^L}, \rho{^R}, \a)$.
 \end{pro}

 \subsection{Dual representation}\label{sub:Dualrep}

 \begin{lem}\label{lem:15}{\em \cite{TS}} Let $(\gg, [,])$ be a Leibniz algebra and $(V, \rho{^L}, \rho{^R})$ be a representation of $(\gg,[,])$. For all  $x\in \gg, v\in V$ and $u^*\in V^*$, define linear maps $\rho^{L*}, \rho^{R*}: \gg\lr \End(V^*)$ by
 \begin{eqnarray}
 &\big\langle\rho^{L*}(x)u^*, v\big\rangle=-\big\langle u^*, \rho^{L}(x)v\big\rangle,&\label{eq:16}\\
 &\big\langle\rho^{R*}(x)u^*, v\big\rangle=-\big\langle u^*, \rho^{R}(x)v\big\rangle.&\label{eq:17}
 \end{eqnarray}
 Then $(V^*, \rho^{L*}, -\rho^{L*}-\rho^{R*})$ is a representation of $(\gg,[,])$, which is called the \textbf{dual representation} of $(V, \rho{^L}, \rho{^R})$.
 \end{lem}

 \begin{lem}\label{lem:18}   Let $(\gg, \rr)$ be a Reynolds Leibniz algebra, $(V, \rho{^L}, \rho{^R})$ be a representation of Leibniz algebra $(\gg,[,])$, $\beta:V\longrightarrow V$ be a linear map. Then the quadruple $\big(V^*, \rho^{L*}, -\rho^{L*}-\rho^{R*}, \beta^*\big)$ is a representation of $(\gg, \rr)$ if and only if for all $x\in \gg, v\in V$, the following two equations hold:
 \begin{eqnarray}
 &\beta\Big(\rho^L\big(x)\beta(v)\Big)+\rho^L(\rr(x))\beta(v)= \beta\Big(\rho^L(\rr (x))v\Big)+\lambda\beta\Big(\rho^L(\rr(x))\beta(v)\Big),&\label{eq:19}\\
 &\beta\Big(\rho^R\big(x)\beta(v)\Big)+\rho^R(\rr(x))\beta(v)= \beta\Big(\rho^R(\rr (x))v\Big)+\lambda\beta\Big(\rho^R(\rr(x))\beta(v)\Big).&\label{eq:20}
 \end{eqnarray}
 \end{lem}

 \begin{proof} By Lemma \ref{lem:15}, $\big(V^*, \rho^{L*}, -\rho^{L*}-\rho^{R*}\big)$ is a representation of $(\gg, [,])$, so we need to check that $\big(V^*, \rho^{L*}, -\rho^{L*}-\rho^{R*}, \beta^*\big)$ meets the Eqs.(\ref{eq:3}) and (\ref{eq:4}) if and only if Eqs.(\ref{eq:19}) and (\ref{eq:20}) hold. According to the Eqs.(\ref{eq:16}) and (\ref{eq:17}), for all $x\in \gg, v\in V, u^*\in V^*$,
 \begin{eqnarray*}
 &&\hspace{-14mm}\Big\langle\rho^{L*}\big(\rr (x)\big)\beta^*(u^*)-{\beta^*}\big(\rho^{L*}(x)\beta^*(u^*)\big)
 -\beta^*\Big(\rho^{L*}\big(\rr(x)\big)u^*\Big)+\lambda\beta^*\big(\rho^{L*}(\rr (x))\beta^*(u^*)\big), v\Big\rangle\\
 &=&\Big\langle u^*, \beta(\rho^L\big(x)\beta(v))+\rho^L\big(\rr (x))\beta(v)-\beta\Big(\rho^L\big(\rr(x)\big)v\Big))
 -\lambda\beta(\rho^L(\rr(x))\beta(v))\Big\rangle.\\
 &&\hspace{-14mm}\Big\langle\big(-\rho^{L*}-\rho^{R*}\big)\big(\rr(x)\big)\beta^*(u^*)
 -{\beta^*}\Big(\big(-\rho^{L*}-\rho^{R*}\big)(x)\beta^*(u^*)\Big)\\
 &&-\beta^*\Big(\big(-\rho^{L*}-\rho^{R*}\big)\big(\rr (x)\big)u^*\Big)+\lambda\beta^*\Big(\big(-\rho^{L*}-\rho^{R*}\big)(\rr (x))\beta^*(u^*)\Big), v\Big\rangle\\
 &=&\Big\langle u^*, \beta\Big(\rho^L\big(\rr(x)\big)v\Big)+\beta\Big(\rho^R(\rr (x))v\Big)-\beta\Big(\rho^L(x)\beta(v)\Big)
 -\beta\big(\rho^R(x))\beta(v)-\rho^L(\rr(x))\beta(v)\\
 &&-\rho^R(\rr(x))\beta(v)+\lambda\beta(\rho^L\big(\rr (x)\big)\beta(v)+\lambda\beta\big(\rho^R(x))\beta(v)\Big\rangle,
 \end{eqnarray*}
 finishing the proof.
 \end{proof}

 Let $(V, \rho{^L}, \rho{^R})=(\gg, L, R)$, then we have

 \begin{cor}\label{cor:23}  Let $(\gg, \rr)$ be a Reynolds Leibniz algebra, $S : \gg\longrightarrow \gg$  a linear map. Then the quadruple $\big(\gg^*, L^{*}, -L^{*}-R^{*}, S^*\big)$ is a representation of $(\gg, \rr)$ if and only if for all $x,y \in \gg$, the following two equations hold:
 \begin{eqnarray}
 &S([x, S(y)])+[\rr(x), S(y)]=S([\rr(x), y])+\lambda S([\rr(x), S(y)]),&\label{eq:24}\\
 &S([S(x), y])+[S(x), \rr(y)]=S([x, \rr(y)])+\lambda S([S(x), \rr(y)]).&\label{eq:25}
 \end{eqnarray}
 \end{cor}

 \begin{defi}\label{de:cl} With notations in Lemma \ref{lem:18} and Corollary \ref{cor:23}, if Eqs.\eqref{eq:19}-\eqref{eq:20} hold, then we say that $\beta$ is \textbf{admissible to $(\gg, \rr)$ with respect to $(V, \rho{^L}, \rho{^R})$}; if Eqs.\eqref{eq:24}-\eqref{eq:25} hold, then we say that $S$ is \textbf{adjoint admissible to $(\gg, \rr)$} or \textbf{$(\gg, \rr)$ is $S$-adjoint admissible}.
 \end{defi}

 \begin{ex}\label{ex:pro:26}  Let $(\gg, \rr)$ be a Reynolds Leibniz algebra. Then
 \begin{enumerate}[(1)]
 \item $\rr^*$  is admissible to $(\gg, \rr)$ with respect to $(\gg^*, L^*, -L^*-R^*)$.
 \item $-\rr$ is adjoint admissible to $(\gg, \rr)$.
 \end{enumerate}
 \end{ex}

 \section{Reynolds Leibniz bialgebras}\label{se:3}

 In this section, we establish a bialgebra theory for Reynolds Leibniz algebras.

 \subsection{Matched pairs of Reynolds Leibniz algebras}

 \begin{lem}\label{lem:56} {\em \cite{MSZ}} Let $(\gg_{1},[,]_{\gg_{1}})$ and $(\gg_{2},[,]_{\gg_{2}})$ be two Leibniz algebras, $\rho_{1}^{L}, \rho_{1}^{R}: \gg_1\lr \End(\gg_2)$ and $\rho_{2}^{L}, \rho_{2}^{R}: \gg_2\lr \End(\gg_1)$ be linear maps. Then $\big(\gg_1, \gg_2, (\rho_1^L, \rho_1^R), (\rho_2^L, \rho_2^R)\big)$ is a matched pair of $(\gg_{1},[,]_{\gg_{1}})$ and $(\gg_{2},[,]_{\gg_{2}})$ if and only if $\gg_1\oplus \gg_2$ is a Leibniz algebra, for all $x, y\in \gg_1, u, v\in \gg_2$, the multiplication is defined as follows:
 \begin{eqnarray}\label{eq:57}
 [x+u, y+v]_{\bowtie}:=[x, y]_{\gg_1}+\rho_{2}^{R}(v)x+\rho_{2}^{L}(u)y+[u, v]_{\gg_2}+\rho_{1}^{R}(y)u+\rho_{1}^{L}(x)v.
 \end{eqnarray}
 \end{lem}

 Having established the basic framework, we now turn to extending the theory of matched pairs of Leibniz algebras to the setting of Reynolds Leibniz algebras.

 \begin{defi}\label{de:61}  Let $(\gg_1, \rr_{\gg_1})$ and $(\gg_2, \rr_{\gg_2})$ be two Reynolds Leibniz algebras. A \textbf{matched pair of $(\gg_1, \rr_{\gg_1})$ and $(\gg_2, \rr_{\gg_2})$} is a quadruple $\big((\gg_1, \rr_{\gg_1}),(\gg_2, \rr_{\gg_2}), (\rho_1^L, \rho_1^R),(\rho_2^L, \rho_2^R)\big)$ such that $(\gg_2, \rho_1^L, \rho_1^R, \rr _{\gg_2})$ is a representation of $(\gg_1, \rr_{\gg_1})$, $(\gg_1, \rho_2^L$, $\rho_2^R, \rr_{\gg_1})$ is a representation of $(\gg_2, \rr_{\gg_2})$ and $\big(\gg_1, \gg_2, (\rho_1^L, \rho_1^R),(\rho_2^L, \rho_2^R)\big)$ is a matched pair of $(\gg_{1}, [,]_{\gg_{1}})$ and $(\gg_{2}$, $[,]_{\gg_{2}})$.
 \end{defi}

 \begin{thm}\label{thm:62}  Let $(\gg_1, \rr_{\gg_1})$, $(\gg_2, \rr_{\gg_2})$ be two Reynolds Leibniz algebras, $\rho_{1}^{L}, \rho_{1}^{R}: \gg_1\lr \End(\gg_2)$ and $\rho_{2}^{L}, \rho_{2}^{R}: \gg_2\lr \End(\gg_1)$ be linear maps. Define a linear map of $\gg_1\oplus \gg_2$ by
 \begin{eqnarray}\label{eq:63}
 \rr_{\gg_1\oplus\gg_2}:\gg_1\oplus \gg_2\lr \gg_1\oplus\gg_2,~~~ \rr _{\gg_1\oplus\gg_2}(x+u):=\rr_{\gg_1}(x)+\rr_{\gg_2}(u), \quad \forall~x\in \gg_1, u\in \gg_2.
 \end{eqnarray}
 Then $\gg_1\oplus \gg_2$ together with the multiplication Eq.(\ref{eq:57}) and linear map Eq.(\ref{eq:63}) is a Reynolds Leibniz algebra if and only if $\big((\gg_1, \rr _{\gg_1}),(\gg_2, \rr_{\gg_2}), (\rho_1^L, \rho_1^R),(\rho_2^L$, $\rho_2^R)\big)$ is a matched pair of $(\gg_1, \rr_{\gg_1})$ and $(\gg_2, \rr_{\gg_2})$.
 \end{thm}

 \begin{proof} For all $x, y\in \gg_1, u, v\in \gg_2$, one has
 \begin{eqnarray*}
 &&\hspace{-10mm}[\rr_{\gg_1\oplus\gg_2}(x+u), \rr_{\gg_1\oplus\gg_2}(y+v)]+\lambda \rr _{\gg_1\oplus\gg_2}\big([\rr_{\gg_1\oplus\gg_2}(x+u), \rr _{\gg_1\oplus\gg_2}(y+v)]\big)\\
 &=&[\rr_{\gg_1}(x),\rr_{\gg_1}(y)]_{\gg_1}+\rho_{2}^{R}(\rr_{\gg_2}(v))\rr _{\gg_1}(x)+\rho_{2}^{L}(\rr_{\gg_2}(u))\rr_{\gg_1}(y)+[\rr_{\gg_2}(u),\rr _{\gg_2}(v)]_{\gg_2}\\
 &&+\rho_{1}^{R}(\rr_{\gg_1}(y))\rr_{\gg_2}(u)+\rho_{1}^{L}(\rr_{\gg_1}(x))\rr _{\gg_2}(v)+\lambda \rr_{\gg_1}([\rr_{\gg_1}(x),\rr_{\gg_1}(y)]_{\gg_1})\\
 &&+\lambda \rr_{\gg_1}(\rho_{2}^{R}(\rr_{\gg_2}(v))\rr_{\gg_1}(x))+\lambda \rr _{\gg_1}(\rho_{2}^{L}(\rr_{\gg_2}(u))\rr_{\gg_1}(y))+\lambda \rr_{\gg_2}([\rr _{\gg_2}(u),\rr_{\gg_2}(v)]_{\gg_2})\\
 &&+\lambda \rr_{\gg_2}(\rho_{1}^{R}(\rr_{\gg_1}(y))\rr_{\gg_2}(u))+\lambda \rr _{\gg_2}(\rho_{1}^{L}(\rr_{\gg_1}(x))\rr_{\gg_2}(v)),\\
 &&\hspace{-10mm}\rr_{\gg_1\oplus\gg_2}([\rr_{\gg_1\oplus\gg_2}(x+u), y+v]+[x+u, \rr _{\gg_1\oplus\gg_2}(y+v)])\\
 &=&\rr_{\gg_1}([\rr_{\gg_1}(x), y]_{\gg_1})+\rr_{\gg_1}(\rho_{2}^{R}(v)\rr _{\gg_1}(x))+\rr_{\gg_1}(\rho_{2}^{L}(\rr_{\gg_2}(u))y)
 +\rr_{\gg_2}([\rr_{\gg_2}(u), v]_{\gg_2})\\
 &&+\rr_{\gg_2}(\rho_{1}^{R}(y)\rr_{\gg_2}(u))+\rr_{\gg_2}(\rho_{1}^{L}(\rr _{\gg_1}(x)v)+\rr_{\gg_1}([x,\rr_{\gg_1}(y)]_{\gg_1})
 +\rr_{\gg_1}(\rho_{2}^{R}(\rr_{\gg_2}(v))x)\\
 &&+\rr_{\gg_1}(\rho_{2}^{L}(u)\rr_{\gg_1}(y))+\rr_{\gg_2}([u, \rr _{\gg_2}(v)]_{\gg_2})+\rr_{\gg_2}(\rho_{1}^{R}(\rr_{\gg_1}(y))u)+\rr _{\gg_2}(\rho_{1}^{L}(x)\rr_{\gg_2}(v)).
 \end{eqnarray*}
 Then $\rr_{\gg_1\oplus\gg_2}$ is a Reynolds operator on $\gg_1\oplus \gg_2$ if and only if $(\gg_2, \rho_1^L, \rho_1^R, \rr_{\gg_2})$ is a representation of $(\gg_1$, $\rr _{\gg_1})$, $(\gg_1, \rho_2^L, \rho_2^R, \rr_{\gg_1})$ is a representation of $(\gg_2, \rr_{\gg_2})$ by comparing the two sides of the equations above. Thus we can finish the proof by Lemma \ref{lem:56}.
 \end{proof}

 \subsection{Manin triple of Reynolds Leibniz algebra}

 First we recall from \cite{TS} the results about skew-symmetric quadratic Leibniz algebras.

 \begin{defi}\label{defi:64} A \textbf{skew-symmetric quadratic Leibniz algebra} $(\gg, \mathfrak{B})$ is a Leibniz algebra $(\gg, [,])$ equipped with a nondegenerate skew-symmetric bilinear form $\mathfrak{B}$ ~\big(in sence of $\mathfrak{B}(x,y)=-\mathfrak{B}(y,x)\big)$, such that the following invariant condition holds:
 \begin{eqnarray}
 \mathfrak{B}(x, [y, z])=\mathfrak{B}([x, z], y)+\mathfrak{B}([z, x], y), \forall~x, y, z\in \gg.\label{eq:65}
 \end{eqnarray}
 \end{defi}

 \begin{rmk}  A skew-symmetric quadratic Leibniz algebra $(\gg, \mathfrak{B})$ satisfies
 \begin{eqnarray}
 \mathfrak{B}(x, [y, z])=-\mathfrak{B}([y, x], z),~~\forall~x, y, z\in \gg. \label{eq:66}
 \end{eqnarray}
 \end{rmk}

 \begin{defi}\label{de:67}\cite{TS} Let $(\gg,[,])$ be a Leibniz algebra. Suppose that there is a Leibniz algebra structure $(\gg^*, [,]_{\gg^*})$ on its dual space $\gg^*$, where $[,]_{\gg^*}$ is the bracket product on $\gg^*$. We construct a Leibniz algebra structure on the direct sum $\gg\oplus \gg^*$ of the underlying vector space of $\gg$ and $\gg^*$ such that $(\gg,[,])$ and $(\gg^*,[,]_{\gg^*})$ are subalgebras and the nondegenerate skew-symmetric bilinear form $\mathfrak{B}_d$ on $\gg\oplus \gg^*$ given by
 \begin{eqnarray}\label{eq:68}
 \mathfrak{B}_d(x+\xi,y+\eta)=\langle \xi, y\rangle-\langle \eta, x\rangle,  \forall~x, y\in\gg, \xi, \eta\in \gg^*,
 \end{eqnarray}
 then $(\gg\oplus \gg^*, \mathfrak{B}_d)$ is a skew-symmetric quadratic Leibniz algebra. Such a construction is called a \textbf{Manin triple of Leibniz algebra associated to $(\gg,[,])$ and $(\gg^*, [,]_{\gg^*})$} and we denote it by $\big((\gg\oplus \gg^*, \mathfrak{B}_d), \gg, \gg^*\big)$. 
 \end{defi}

 Now we extend the concept above to Reynolds Leibniz algebras.

 \begin{defi}\label{defi:69}  A \textbf{skew-symmetric quadratic Reynolds Leibniz algebra} is a triple $(\gg, \rr,  \mathfrak{B})$, where $(\gg, \rr)$ is a Reynolds Leibniz algebra, and $(\gg, \mathfrak{B})$ is a skew-symmetric quadratic Leibniz algebra.
 \end{defi}

 Let $\widehat{\rr}$ denote the adjoint operator of $\rr$ in the following way: 
 \begin{eqnarray}\label{eq:70}
 \mathfrak{B}\big(\rr(x), y\big)=\mathfrak{B}\big(x, \widehat{\rr}(y)\big),~ \forall~ x, y\in \gg.
 \end{eqnarray}

 \begin{pro}\label{pro:71}  Let $(\gg, \rr, \mathfrak{B})$  be a skew-symmetric quadratic Reynolds Leibniz algebra, then $(\gg^*, L^*, -L^*-R^*, {\widehat{\rr}}^*)$ is a representation of $(\gg, \rr)$, and further, it is equivalent to $(\gg, L, R, \rr)$ as representations of $(\gg, \rr)$.

 Conversely, let $(\gg, \rr)$ be a Reynolds Leibniz algebra, $S:\gg\lr \gg$ be adjoint admissible to $(\gg, \rr)$. If $(\gg^*, L^*, -L^*-R^*, S^*)$ is equivalent to $(\gg, L, R, \rr)$ as representations of $(\gg, \rr)$, then there is a nondegenerate invariant bilinear $\mathfrak{B}$ such that $(\gg, \rr,  \mathfrak{B})$ is a quadratic Reynolds Leibniz algebra and $\widehat{\rr}=S$.
 \end{pro}

 \begin{proof} For all $x, y, z\in \gg$, by Eq.(\ref{eq:2}), we obtain
 \begin{eqnarray*}
 \mathfrak{B}\big(x, [\rr(y), \rr(z)]\big)+\mathfrak{B}\big(x, \lambda \rr([\rr(y),\rr (z)])\big)&=&\mathfrak{B}\big(x, \rr([\rr(y), z])\big)+\mathfrak{B}\big(x, \rr[y, \rr (z)]\big).\\
 &\Updownarrow&\\
 -\mathfrak{B}\big([\rr(y), x], \rr(z)\big)-\lambda\mathfrak{B}\big([\rr(y), \widehat{\rr }(x)], \rr(z)\big)&=&-\mathfrak{B}\big([y, \widehat{\rr}(x)], \rr (z)\big)-\mathfrak{B}\big([\rr(y), \widehat{\rr}(x)], z\big).\\
 &\Updownarrow&\\
 \mathfrak{B}\big(\widehat{\rr}([\rr(y), x]), z\big)+\mathfrak{B}\big(\lambda\widehat{\rr }[\rr(y), {\widehat{\rr}}(x)], z\big)&=&\mathfrak{B}\big(\widehat{\rr}[y, \widehat{\rr }(x)], z\big)+\mathfrak{B}\big([\rr(y), \widehat{\rr}(x)], z\big).
 \end{eqnarray*} 
 Then $\widehat{\rr}([\rr(y), x])+\lambda\widehat{\rr}([\rr(y), {\widehat{\rr}}(x)])=[\rr (y), \widehat{\rr}(x)]+\widehat{\rr}([y, \widehat{\rr}(x)])$ by the nondegeneracy of $\mathfrak{B}$. So Eq.(\ref{eq:24}) holds for $\widehat{\rr}$. Similarly,
 \begin{eqnarray*}
 0&\stackrel{(\ref{eq:2})}{=}&\mathfrak{B}\big(x, [\rr(y), \rr(z)]+\lambda \rr([\rr (y),\rr(z)])-\rr([\rr(y), z])-\rr([y, \rr(z)])\big)\\
 &\stackrel{(\ref{eq:65})(\ref{eq:70})}{=}&\mathfrak{B}\big([x, \rr(z)]+[\rr(z),x], \rr (y)\big)+\lambda \mathfrak{B}\big([\widehat{\rr}(x), \rr(z)]+[\rr(z), \widehat{\rr}(x)], \rr(y)\big)\\
 &&-\mathfrak{B}\big([\widehat{\rr}(x), \rr(z)]+[\rr(z), \widehat{\rr}(x)], y\big)-\mathfrak{B}\big([\widehat{\rr}(x), z]+[z, \widehat{\rr}(x)], \rr(y)\big)\\
 &\stackrel{(\ref{eq:70})}{=}&\mathfrak{B}\big(\widehat{\rr}([x, \rr(z)])+\widehat{\rr }([\rr(z),x]), y\big)+\mathfrak{B}\big(\lambda\widehat{\rr}([\widehat{\rr}(x), \rr (z)])+\lambda\widehat{\rr}([\rr(z), \widehat{\rr}(x)]), y\big)\\
 &&-\mathfrak{B}\big([\widehat{\rr}(x), \rr(z)]+[\rr(z), \widehat{\rr}(x)], y\big)-\mathfrak{B}\big(\widehat{\rr}([\widehat{\rr}(x), z])+\widehat{\rr}([z, \widehat{\rr}(x)]), y\big)\\
 &\stackrel{(\ref{eq:24})}{=}&\mathfrak{B}\big(\widehat{\rr}([x, \rr(z)])+\lambda \widehat{\rr}([\widehat{\rr}(x), \rr(z)])-\widehat{\rr}([\widehat{\rr}(x), z])-[\widehat{\rr}(x), \rr(z)], y\big).
 \end{eqnarray*}
 Hence Eq.(\ref{eq:25}) holds, $(\gg^*, L^*, -L^*-R^*, {\widehat{\rr}}^*)$ is a representation of $(\gg, \rr)$. Define a linear map $\phi:\gg\lr \gg^* $:
 \begin{eqnarray*}\label{eq:71}
 \phi(x)y:=\langle\phi(x),y\rangle=\mathfrak{B}(x,y), \forall~x, y\in \gg.
 \end{eqnarray*}
 The nondegeneracy of $\mathfrak{B}$ gives the bijectivity of $\phi$.
 Then by \cite[Proposition 3.9]{MSZ},
 $(\gg, L, R, \rr)$ is equivalent to $(\gg^*, L^*, -L^*-R^*, {\widehat{\rr}}^*)$ as representations of $(\gg, \rr)$.

 Conversely, suppose that $\phi:\gg\lr \gg^*$ is the linear isomorphism giving the equivalence between $(\gg, L, R, \rr)$ and $(\gg^*, L^*, -L^*-R^*, {\widehat{\rr}}^*)$. Define a bilinear form $\mathfrak{B}$ on $\gg$ by $\mathfrak{B}(x, y):=\big\langle\phi(x), y\big\rangle,~\forall~x, y\in \gg$. Then we can obtain the quadratic Reynolds Leibniz algebra $(\gg, \rr,  \mathfrak{B})$ and $\widehat{\rr}=S$ by a similar argument given in the proof of necessity.
 \end{proof}

 We now extend the notion of Manin triple to Reynolds Leibniz algebras.
 \begin{defi}\label{de:72}  Let $((\gg, [,]), \rr)$ and $(\gg^*, [,]_{\gg^*}, S^*)$ be two Reynolds Leibniz algebras. A \textbf{Manin triple of Reynolds Leibniz algebra associated to $((\gg, [,]), \rr)$ and $(\gg^*, [,]_{\gg^*}, S^*)$}   is a Manin triple $((\gg \oplus \gg^*, \mathfrak{B}_d), \gg, \gg^*)$ of Leibniz algebra associated to $(\gg, [,])$ and $(\gg^*,[,]_{\gg^*})$ such that $(\gg\oplus \gg^*, \rr+S^*, \mathfrak{B}_d)$ is a skew-symmetric quadratic Reynolds Leibniz algebra. We denote it by $(\gg\oplus \gg^*, \rr +S^*, \mathfrak{B}_d)$.
 \end{defi}

 \begin{lem}\label{lem:73}
 Let $(\gg\oplus \gg^*, \rr+S^*, \mathfrak{B}_d)$ be a Manin triple of Reynolds Leibniz algebra associated to $((\gg, [,]), \rr)$ and $(\gg^*, [,]_{\gg^*}, S^*)$. Then
 \begin{enumerate}[(1)]
 \item \label{it:74} The adjoint $\widehat{\rr+S^*}$ of $\rr+S^*$ with respect to $\mathfrak{B}_d$ is $S+\rr^*$. Furthermore, $S+\rr^*$ is adjoint admissible to $(\gg\oplus\gg^*,\rr+S^*)$.
 \item \label{it:75} $S$ is adjoint admissible to $(\gg,\rr)$.
 \item \label{it:76} $\rr^*$ is adjoint admissible to $(\gg^*,S^*)$.
 \end{enumerate}
 \end{lem}

 \begin{proof}
 \ref{it:74} For all $x,y \in \gg, \xi, \eta\in \gg^*$, we have
 \begin{eqnarray*}
 \mathfrak{B}_d\big((\rr+S^*)(x+\xi), y+\eta\big)
 &=&\mathfrak{B}_d\big(\rr(x)+S^*(\xi), y+\eta\big)\stackrel {(\ref{eq:68})}{=}\big\langle S^*(\xi), y\big\rangle-\big\langle \eta, \rr(x)\big\rangle\\
 &=&\big\langle \xi, S(y)\big\rangle-\big\langle \rr^*(\eta), x\big\rangle=\mathfrak{B}_d\big(x+\xi, (S+\rr^*)(y+\eta)\big).
 \end{eqnarray*}
 Hence $\widehat{\rr+S^*}=S+\rr^*$. In addition, by Proposition \ref{pro:71}, $S+\rr^*$ is adjoint admissible to $(\gg\oplus\gg^*,\rr+S^*)$.

 \ref{it:75} By Item \ref{it:74}, $S+\rr^*$ is adjoint admissible to $(\gg\oplus\gg^*,\rr +S^*)$, then according to the Eqs.(\ref{eq:24}) and (\ref{eq:25}), for all $x, y \in \gg, \xi, \eta\in \gg^*$, the following two equations hold. 
 \begin{eqnarray}\label{eq:77}
 &&\hspace{-16mm}(S+\rr^*)\Big([x+\xi, S(y)+\rr^*(\eta)]\Big)+[\rr(x)+S^*(\xi),S(y)+{\rr ^*}(\eta)]\nonumber\\
 &=&(S+\rr^*)\Big([\rr(x)+S^*(\xi), y+\eta]\Big)+\lambda(S+\rr^*)\Big([\rr(x)+S^*(\xi), S(y)+\rr^*(\eta)]\Big)
 \end{eqnarray}
 and
 \begin{eqnarray}\label{eq:78}
 &&\hspace{-16mm}(S+\rr^*)\Big([S(x)+{\rr^*}(\xi), y+\eta]\Big)+[S(x)+{\rr^*}(\xi), \rr (y)+S^*(\eta)]\nonumber\\
 &=&(S+\rr^*)\Big([x+\xi, \rr(y)+S^*(\eta)]\Big)+\lambda(S+\rr^*)\Big([S(x)+\rr^*(\xi), \rr(y)+S^*(\eta)\Big).
 \end{eqnarray}
 Setting $\xi=\eta=0$ in Eqs.(\ref{eq:77}) and (\ref{eq:78}), we obtain Item \ref{it:75}. 

 \ref{it:76} Item \ref{it:76} can be gotten by setting $x=y=0$ in Eqs.(\ref{eq:77}) and (\ref{eq:78}).
 \end{proof}

 \begin{lem}\label{lem:79}{\em \cite{TS}} Let $(\gg,[,])$ and $(\gg^*, [,]_{\gg^*})$ be two Leibniz algebras. Then the following conditions are equivalent:
 \begin{enumerate}[(1)]
 \item \mlabel{it:2} $\big(\gg, \gg^*, (L^*_\gg, -L^*_\gg-R^*_\gg), (L^*_{\gg^*}, -L^*_{\gg^*}-R^*_{\gg^*})\big)$ is a matched pair of $(\gg,[,])$ and $(\gg^*, [,]_{\gg^*})$,
 \item \mlabel{it:3} $\big((\gg\oplus \gg^*, \mathfrak{B}_d), \gg, \gg^*\big)$ is a Manin triple of Leibniz algebra associated to $(\gg,[,])$ and $(\gg^*, [,]_{\gg^*})$.
 \end{enumerate}
 \end{lem}

 \begin{thm}\label{thm:80}  Let $((\gg, [,]), \rr)$ and $(\gg^*, [,]_{\gg^*}, S^*)$ be two Reynolds Leibniz algebras. Then there is a Manin triple of Reynolds Leibniz algebra $(\gg\oplus \gg^*, \rr+S^*, \mathfrak{B}_d)$ associated to $(\gg,$ $[,], \rr)$ and $(\gg^*, [,]_{\gg^*}, S^*)$ if and only if $\big((\gg, \rr), (\gg^*, S^*), (L^*_\gg, -L^*_\gg-R^*_\gg), (L^*_{\gg^*}, -L^*_{\gg^*}-R^*_{\gg^*})\big)$ is a matched pair of $((\gg, [,]), \rr)$ and $(\gg^*, [,]_{\gg^*}, S^*)$.
 \end{thm}

 \begin{proof}
 $(\Longrightarrow)$ Let $(\gg\oplus \gg^*, \rr+S^*, \mathfrak{B}_d)$ be a Manin triple of a Reynolds Leibniz algebra associated to $((\gg, [,]), \rr)$ and $(\gg^*, [,]_{\gg^*}, S^*)$, then $\big((\gg\oplus \gg^*, \mathfrak{B}_d), \gg, \gg^*\big)$ is a Manin triple of Leibniz algebra associated to $(\gg, [,])$ and $(\gg^*, [,]_{\gg^*})$. By Lemma \ref{lem:79}, $\big(\gg,\gg^*, (L^*_\gg, -L^*_\gg-R^*_\gg), (L^*_{\gg^*}, -L^*_{\gg^*}-R^*_{\gg^*})\big)$ is a matched pair of $(\gg,[,])$ and $(\gg^*, [,]_{\gg^*})$. Furthermore, by Lemma \ref{lem:73}, $\big(\gg^*,L^*_\gg, -L^*_\gg-R^*_\gg, S^*)$ is a representation of ($\gg,\rr$) and $\big(\gg, L^*_{\gg^*}, -L^*_{\gg^*}-R^*_{\gg^*},\rr)$ is a representation of ($\gg^*, S^*$). Hence,$\big((\gg, \rr), (\gg^*, S^*), (L^*_\gg, -L^*_\gg-R^*_\gg), (L^*_{\gg^*}, -L^*_{\gg^*}-R^*_{\gg^*})\big)$ is a matched pair of $((\gg, [,]), \rr)$ and $(\gg^*, [,]_{\gg^*}, S^*)$.

 $(\Longleftarrow)$ If $\big((\gg, \rr), (\gg^*, S^*), (L^*_\gg, -L^*_\gg-R^*_\gg), (L^*_{\gg^*}, -L^*_{\gg^*}-R^*_{\gg^*})\big)$ is a matched pair of $((\gg, [,]), \rr)$ and $(\gg^*, [,]_{\gg^*}, S^*)$, then $\big((\gg, \gg^*, (L^*_\gg, -L^*_\gg-R^*_\gg), (L^*_{\gg^*}, -L^*_{\gg^*}-R^*_{\gg^*})\big)$ is a matched pair of $(\gg, [,])$ and $(\gg^*, [,]_{\gg^*})$. By Lemma \ref{lem:79}, \mlabel{it:3} $\big((\gg\oplus \gg^*, \mathfrak{B}_d), \gg, \gg^*\big)$ is a Manin  triple of Leibniz algebra associated to $(\gg,[,])$ and $(\gg^*, [,]_{\gg^*})$. Furthermore, $(\gg_1\oplus \gg_2, \rr+S^*)$ is a Reynolds Leibniz algebra, by Theorem \ref{thm:62}, $(\gg\oplus \gg^*, \rr+S^*, \mathfrak{B}_d)$ is a Manin triple of Reynolds Leibniz algebra associated to $(\gg,$ $[,], \rr)$ and $(\gg^*, [,]_{\gg^*}, S^*)$. 
 \end{proof}

 \subsection{Reynolds Leibniz bialgebras} In this subsection, we introduce the notion of a Reynolds Leibniz bialgebra and establish its equivalence with a matched pair of Reynolds Leibniz algebras. We begin by recalling the definitions of a Leibniz coalgebra and a Leibniz bialgebra from \cite{TS}.

 A {\bf Leibniz coalgebra} is a pair $(\gg,\delta)$, where $\gg$ is a linear space, $\delta:\gg\lr \gg\o \gg$ is a linear map, such that the following condition holds:
 \begin{eqnarray*}
 (\id\o \delta)\delta=(\delta\o\id)\delta+(\tau\o \id)(\id\o \delta)\delta.\label{eq:28}
 \end{eqnarray*}

 Let $(\gg,[,])$ be a Leibniz algebra and $(\gg, \d)$ be a Leibniz coalgebra. If for all $x, y\in \gg$, the following equations hold:
 \begin{eqnarray*}
 &(R_x\o \id)\delta(y)=\tau\big((R_y\o \id)\delta(x)\big),\label{eq:82}\\
 &\delta\big([x, y]\big)=\big((\id\o R_y-L_y\o \id-R_y\o \id)\circ (\id+ \tau)\big)\delta(x)+(\id\o L_x+L_x \o \id)\delta(y).\label{eq:83}
 \end{eqnarray*}
 Then we call $(\gg, [,], \d)$ a {\bf Leibniz bialgebra}.

 Now we extend this notion to Reynolds Leibniz bialgebra, we have

 \begin{defi}\label{defi:29}  A {\bf Reynolds Leibniz coalgebra} is a triple $((\gg, \delta), S)$ including a Leibniz coalgebra $(\gg, \delta)$ and a Reynolds operator $S$ on $(\gg, \delta)$, i.e., a linear map $S: \gg\lr \gg$ satisfying:
 \begin{eqnarray}
 (S\o S)\delta+\lambda(S\o S)\delta S=(S\o \id)\delta S+(\id\o S)\delta S. \label{eq:30}
 \end{eqnarray}
 \end{defi}

 \begin{defi}\label{de:87}  A {\bf Reynolds Leibniz bialgebra} is a vector space $\gg$ together with linear maps $[,]: \gg\o \gg\lr \gg$, $\delta : \gg\lr \gg\o \gg$, $\rr,  S: \gg\lr \gg$ such that:
 \begin{enumerate}[(1)]
 \item \label{it:de:87a} $(\gg,[,],\delta)$ is a Leibniz bialgebra.
 \item \label{it:de:87b} $((\gg,[,]),\rr)$ is a Reynolds Leibniz algebra.
 \item \label{it:de:87c}$((\gg,\delta),S)$ is a Reynolds Leibniz coalgebra.
 \item \label{it:de:87d}$S$ is adjoint admissible to $((\gg,[,]),\rr)$.
 \item \label{it:de:87e} $\rr^*$ is adjoint admissible to $(\gg^*,\delta^*,S^*)$, that is, the following conditions hold:
 \begin{eqnarray}
 &(\id \o \rr)\delta \rr+(S\o \rr)\delta=(S\o \id)\delta \rr+\lambda (S \o \rr)\delta \rr, &\label{eq:88}\\
 &(\rr\o \id)\delta \rr+(\rr\o S)\delta=(\id\o S)\delta \rr+\lambda (\rr\o S)\delta \rr .&\label{eq:89}
 \end{eqnarray}
 \end{enumerate}
 We use $(\gg,[,],\delta,\rr, S)$ (abbr. $((\gg, \rr), \delta, S)$) to denote a Reynolds Leibniz bialgebra.
 \end{defi}

 \begin{rmk} By Example \ref{ex:pro:26}, when $S=-\rr$, then $(\gg, [,], \delta, \rr,  -\rr)$ is a Reynolds Leibniz bialgebra if and only if items \ref{it:de:87a}-\ref{it:de:87c} hold.
 \end{rmk}

 Let $\l=0$ in Definition \ref{de:87}, one can obtain

 \begin{tdef}\label{thm:de:87} A {\bf Rota-Baxter Leibniz bialgebra of weight 0} is a vector space $\gg$ together with linear maps $[,]: \gg\o \gg\lr \gg$, $\delta : \gg\lr \gg\o \gg$, $\rr,  S: \gg\lr \gg$ such that:
 \begin{enumerate}[(1)]
 \item $(\gg,[,],\delta)$ is a Leibniz bialgebra.
 \item $((\gg,[,]),\rr)$ is a Rota-Baxter Leibniz algebra of weight 0.
 \item $((\gg,\delta),S)$ is a Rota-Baxter Leibniz coalgebra of weight 0.
 \item the following conditions hold:
 \begin{eqnarray*}
 &S([x, S(y)])+[\rr(x), S(y)]=S([\rr(x), y]),&\label{eq:24-1}\\
 &S([S(x), y])+[S(x), \rr(y)]=S([x, \rr(y)]),&\label{eq:25-1}\\
 &(\id \o \rr)\delta \rr+(S\o \rr)\delta=(S\o \id)\delta \rr, &\label{eq:88-1}\\
 &(\rr\o \id)\delta \rr+(\rr\o S)\delta=(\id\o S)\delta \rr.&\label{eq:89-1}
 \end{eqnarray*}
 \end{enumerate}
 \end{tdef}

 \begin{lem}\label{lem:84}{\em \cite{TS}}  Let $(\gg,[,])$ and $(\gg^*, \delta^*)$ be Leibniz algebras. Then the following conditions are equivalent:
 \begin{enumerate}[(1)]
 \item \mlabel{it:85} $(\gg, [,], \d)$ is a Leibniz bialgebra;
 \item \mlabel{it:86} $\big(\gg, \gg^*, (L^*_\gg, -L^*_\gg-R^*_\gg), (L^*_{\gg^*}, -L^*_{\gg^*}-R^*_{\gg^*})\big)$ is a matched pair of $(\gg,[,])$ and $(\gg^*, \delta^*)$.
 \end{enumerate}
 \end{lem}

 For Reynolds Leibniz algebras, one can get
 \begin{thm}\label{thm:91}  Let $((\gg, [,]), \rr)$ be a Reynolds Leibniz algebra. Suppose that there is a Reynolds Leibniz algebra structure $(\gg^*,\delta^*,S^*)$ on its dual space. Then the quintuple $(\gg,[,],\delta,$  $\rr, S)$ is a Reynolds Leibniz bialgebra if and only if the quadruple~$\big((\gg, \rr), (\gg^*, S^*),(L^*_\gg, -L^*_\gg-R^*_\gg), (L^*_{\gg^*}, -L^*_{\gg^*}-R^*_{\gg^*})\big)$ is a matched pair of $((\gg, [,]), \rr)$ and $((\gg^*, \delta^*), S^*)$.
 \end{thm}

 \begin{proof}
 $(\Longrightarrow)$ If $(\gg,[,],\delta, \rr,  S)$ is a Reynolds Leibniz bialgebra, then $(\gg, [,], \delta)$ is a Leibniz bialgebra, $S$ and $\rr^*$ are adjoint admissible to $((\gg, [,]), \rr)$ and $((\gg^*, \delta^*), S^*)$ respectively. According to the Lemma \ref{lem:84},  $(\gg, [,], \delta)$ is a Leibniz bialgebra means $\big(\gg,\gg^*, (L^*_\gg,-L^*_\gg-R^*_\gg), (L^*_{\gg^*}, -L^*_{\gg^*}-R^*_{\gg^*})\big)$ is a matched pair of $(\gg, [,])$ and $(\gg^*, \delta^*)$, $S$ and $\rr^*$ are adjoint admissible to $((\gg, [,]), \rr)$ and $((\gg^*, \delta^*), S^*)$ means $(\gg^*, L^*, -L^*-R^*, S^*)$ is a representation of $(\gg, \rr)$, $\big(\gg, L^*_{\gg^*}, -L^*_{\gg^*}-R^*_{\gg^*},\rr)$ is a representation of $(\gg^*, S^*)$. Hence $\big((\gg, \rr), (\gg^*, S^*), (L^*_\gg, -L^*_\gg-R^*_\gg), (L^*_{\gg^*},$ $-L^*_{\gg^*}-R^*_{\gg^*})\big)$ is a matched pair of $((\gg, [,]), \rr)$  and $((\gg^*, \delta^*),$ $S^*)$.\\
 $(\Longleftarrow)$ If $\big((\gg, \rr), (\gg^*, S^*), (L^*_\gg, -L^*_\gg-R^*_\gg), (L^*_{\gg^*}, -L^*_{\gg^*}-R^*_{\gg^*})\big)$ is a matched pair of $((\gg, [,]), \rr)$  and $((\gg^*, \delta^*), S^*)$, then $\big(\gg,\gg^*, (L^*_\gg, -L^*_\gg-R^*_\gg), (L^*_{\gg^*}, -L^*_{\gg^*}-R^*_{\gg^*})\big)$ is a matched pair of $(\gg,[,])$ and $(\gg^*, \delta^*)$ and $S$, $\rr^*$ are adjoint admissible to $((\gg,[,]),\rr)$ and $((\gg^*,\delta^*),S^*)$ respectively. Hence $(\gg,[,],\delta,\rr, S)$ is a Reynolds Leibniz bialgebra.
 \end{proof}

 Combining Theorem \ref{thm:80} and \ref{thm:91}, we have
 \begin{thm}\label{thm:92}  Let $((\gg, [,]), \rr)$ and $((\gg^*, \delta^*), S^*)$ be two Reynolds Leibniz algebras. Then the conditions are equivalent:
 \begin{enumerate}[(1)]
 \item \label{it:a1} $\big((\gg, \rr), (\gg^*, S^*),(L^*_\gg, -L^*_\gg-R^*_\gg), (L^*_{\gg^*}, -L^*_{\gg^*}-R^*_{\gg^*})\big)$ is a matched pair of $((\gg, [,]), \rr)$ and $((\gg^*, \delta^*), S^*)$;
 \item \label{it:b1} There is a Manin triple $(\gg\oplus\gg^*, \rr+S^*, \mathfrak{B}_d)$ of Reynolds Leibniz algebra associated to $((\gg, [,]), \rr)$ and $((\gg^*, \delta^*), S^*)$;
 \item \label{it:c1} $(\gg, [,], \delta, \rr,  S)$ is a Reynolds Leibniz bialgebra.
 \end{enumerate}
 \end{thm}

 \section{Triangular Reynolds Leibniz bialgebras, admissible cLYBes and $\mathcal{O}$-operators}\label{se:4}

 In this section, we develop several constructions of Reynolds Leibniz bialgebras by employing both admissible cLYBes and $\mathcal{O}$-operators.

 \subsection{Triangular Reynolds Leibniz bialgebras}

 \begin{thm}\label{thm:93} {\em \cite{LMW}} Let $(\gg, [,])$ be a Leibniz algebra, $r\in \gg\o \gg$. Then $(\gg, [,], \delta_{r})$ where $\delta_{r}$ is defined by
 \begin{eqnarray}\label{eq:94}
 \delta(x):=\delta_{r}(x)=-r^1\o [r^2, x]+[r^2, x]\o r^1+[x, r^2]\o r^1, \forall~x\in \gg,
 \end{eqnarray}
 is a Leibniz bialgebra if and only if the following conditions hold:
 \begin{eqnarray}
 &(R_x\o R_y)(r^{\tau}-r)=0,&\label{eq:95}\\
 &(L_x\o R_y+R_y\o L_x+L_y\o L_x)(r^{\tau}-r)=0,&\label{eq:96}\\
 &\hspace{-10mm}(\id\o L_x\o \id+\id\o R_x\o \id)(r_{12}r_{23}^{\tau}+r_{13}r_{23}^{\tau}-r_{12}r_{13}^{\tau}-r_{13}^{\tau}r_{12})\nonumber&\\
 &\hspace{-33mm}-(\id\o \id \o R_x)(r_{12}r_{23}+r_{13}r_{23}-r_{12}^{\tau}r_{13}-r_{13}r_{12}^{\tau})\nonumber&\\
 &-(L_x\o \id\o \id+R_x\o \id\o \id)(r_{23}r_{13}^{\tau}+r_{12}^{\tau}r_{13}^{\tau}-r_{23}^{\tau}r_{12}^{\tau}-r_{12}^{\tau}r_{23}^{\tau})=0,\label{eq:97}&
 \end{eqnarray}
 where
 \begin{eqnarray*}
 r_{12}r_{23}=r^1\o [r^2, \bar{r}^1]\o \bar{r}^2, r_{13}r_{23}=r^1\o \bar{r}^1\o [r^2, \bar{r}^2], r_{23}r_{12}=r^1\o [\bar{r}^1, r^2]\o \bar{r}^2,\\
 r_{23}r_{13}=r^1\o \bar{r}^1\o [\bar{r}^2, r^2], r_{12}r_{13}=[r^1, \bar{r}^1]\o r^2\o \bar{r}^2, r_{13}r_{12}=[r^1, \bar{r}^1]\o \bar{r}^2 \o r^2,
 \end{eqnarray*}
 $r=r^1\o r^2, r^{\tau}=\tau\circ r=r^2\o r^1,$ and $ \bar{r}=r$.
 \end{thm}

 \begin{defi} A Leibniz bialgebra $(\gg, [,], \delta)$ is called {\bf coboundary} if there exists an element $r\in \gg\o \gg$ such that Eq.(\ref{eq:94}) holds.
 \end{defi}

 \begin{cor}\label{cor:98}{\em \cite{{LMW},{TS}}} Let $(\gg,[,])$ be a Leibniz algebra. If $r\in \gg\o \gg$ is the symmetric solution of the following classical Leibniz Yang-Baxter equation (cLYBe) in $(\gg,[,])$,
 \begin{eqnarray}\label{eq:99}
 &r_{12}r_{23}+r_{13}r_{23}=r_{12}^{\tau}r_{13}+r_{13}r_{12}^{\tau},&\label{eq:100}
 \end{eqnarray}
 then $(\gg,[,], \delta_r)$ is a Leibniz bialgebra, where $\delta_{r}$ is defined by Eq.(\ref{eq:94}). In this case, we call this Leibniz bialgebra \textbf{triangular}, denoted by $(\gg, [,], \delta_r)$.
 \end{cor}

 By Theorem \ref{thm:93}, under the assumption of $S$-adjoint admissible Reynolds Leibniz algebra $(\gg, \rr)$, if $r\in \gg\o \gg$ is a symmetric solution of the cLYBe in $(\gg,[,])$, in addition, $((\gg, \delta), S)$ is a Reynolds Leibniz coalgebra and Eqs.(\ref{eq:88}) and (\ref{eq:89}) hold, then $\big((\gg, \rr), \delta, S\big)$ is a Reynolds Leibniz bialgebra.

 The following result establishes a fundamental link between solutions of the cLYBe and Reynolds operators on a Leibniz algebra and a Leibniz coalgebra.

 \begin{thm}\label{thm:102}  Let $(\gg, \rr)$ be an $S$-adjoint admissible Reynolds Leibniz algebra and define a linear map $\delta:\gg\lr \gg\o \gg$ by Eq.(\ref{eq:94}), then we have
 \begin{enumerate}[(1)]
 \item \label{it:103} Eq.(\ref{eq:30}) holds if and only if, for all $x\in \gg$,
 \begin{eqnarray}\label{eq:104}
 &&(R(S(x))\o \id+L(S(x))\o \id-(S\circ R(x))\o \id-(S\circ L(x))\o \id-\lambda (S\circ R)(S(x))\nonumber\\
 &&\o \id-\lambda (S\circ L(S(x)))\o \id)(\rr\o \id-\id \o S)(r^{\tau})\nonumber+(\id\o R(S(x))\\
 &&-\lambda\id\o (S\circ R(S(x)))-\id\o (S\circ R(x)))(S\o \id-\id\o \rr)(r)=0.
 \end{eqnarray}

 \item \label{it:105} Eq.(\ref{eq:88}) holds if and only if, for all $x\in \gg$,
 \begin{eqnarray}\label{eq:106}
 &&(\id \o R(\rr(x))-\id\o (\rr\circ R(x)) +\lambda \id \o (\rr\circ  R(\rr(x))))(S\o \id-\id \o \rr)(r)\nonumber\\
 &&+(\lambda (S\circ R(\rr(x)))\o \id+\lambda (S\circ L(\rr(x))) \o \id-(S\circ R(x))\o \id-R(\rr(x))\o \id\nonumber\\
 &&-L(\rr(x)) \o \id-(S\circ L(x))\o \id)(S\o \id-\id \o \rr)(r^{\tau})=0.
 \end{eqnarray}

 \item \label{it:107} Eq.(\ref{eq:89}) holds if and only if, for all $x\in \gg$,
 \begin{eqnarray}\label{eq:108}
 &&(\id\o R(\rr(x))+\id\o (S\circ R(x))-\lambda \id\o (S\circ R(\rr(x))))(\id \o S-\rr\o \id)(r)\nonumber\\
 &&+((\rr\circ L(x))\o \id+(\rr\circ R(x))\o \id-R(\rr(x))\o \id-L(\rr(x))\o \id\nonumber\\
 &&-\lambda (\rr\circ R(\rr(x)))\o \id-\lambda (\rr\circ L(\rr(x)))\o \id)(\id \o S-\rr\o \id)(r^{\tau})=0.
 \end{eqnarray}
 \end{enumerate}
 \end{thm}

 \begin{proof} \ref{it:103} For all $x\in \gg$, by Eq.(\ref{eq:94}) we have
 \begin{eqnarray*}
 (S\o S)\delta(x)\hspace{-3mm}&=&\hspace{-3mm}-S(r^1)\o S([r^2, x])+S([r^2, x])\o S(r^1)+S([x, r^2])\o S(r^1),\\
 \lambda (S\o S)\delta S(x)\hspace{-3mm}&=&\hspace{-3mm}-\lambda S(r^1)\o S([r^2, S(x)])+\lambda S([r^2, S(x)])\o S(r^1)\\
 &&+\lambda S([S(x), r^2])\o S(r^1),\\
 (S\o \id)\delta (S(x))\hspace{-3mm}&=&\hspace{-3mm}-S(r^1)\o [r^2, S(x)]+S([r^2, S(x)])\o r^1+S([S(x), r^2])\o r^1\\
 &\stackrel {(\ref{eq:24})(\ref{eq:25})}{=}&-S(r^1)\o [r^2, S(x)]+S([\rr(r^2), x])\o r^1+\lambda S([\rr(r^2), S(x)])\o r^1\\
 &&-[\rr(r^2), S(x)]\o r^1+S([x, \rr(r^2)])\o r^1+\lambda S([S(x), \rr(r^2)])\o r^1\\
 &&-[S(x),\rr(r^2)]\o r^1,\\
 (\id\o S)\delta (S(x))\hspace{-3mm}&=&\hspace{-3mm}-r^1\o S([r^2, S(x)])+[r^2, S(x)]\o S(r^1)+[S(x), r^2]\o S(r^1)\\
 &\stackrel {(\ref{eq:24})}{=}&-r^1\o S([\rr(r^2), x])-r^1\o \lambda S([\rr(r^2), S(x)])+r^1\o [\rr(r^2), S(x)]\\
 &&+[r^2, S(x)]\o S(r^1)+[S(x), r^2]\o S(r^1),
 \end{eqnarray*}
 then Eq.(\ref{eq:30}) $\Longleftrightarrow$  Eq.(\ref{eq:104}).

 \ref{it:105} For all $x\in \gg$, by Eq.(\ref{eq:94}) we have
 \begin{eqnarray*}
 (\id\o \rr)\delta(\rr(x))\hspace{-3mm}&=&\hspace{-3mm}-r^1\o \rr([r^2, \rr(x)])+[r^2, \rr (x)]\o \rr(r^1)+[\rr(x), r^2]\o \rr(r^1)\\
 &\stackrel {(\ref{eq:2})}{=}&-r^1\o [\rr(r^2), \rr(x)]+r^1\o \rr([\rr(r^2),x])-\lambda r^1\o \rr([\rr(r^2), \rr(x)])\\
 &&+[r^2, \rr(x)]\o \rr(r^1)+[\rr(x), r^2]\o \rr(r^1),\\
 (S\o \rr)\delta(x)\hspace{-3mm}&=&\hspace{-3mm}-S(r^1)\o \rr([r^2, x])+S([r^2, x])\o \rr (r^1)+S([x, r^2])\o \rr(r^1),\\
 (S\o \id)\delta(\rr(x))\hspace{-3mm}&=&\hspace{-3mm}-S(r^1)\o [r^2, \rr(x)]+S([r^2, \rr (x)])\o r^1+S([\rr(x), r^2])\o r^1\\
 &\stackrel {(\ref{eq:24})(\ref{eq:25})}{=}&-S(r^1)\o ([r^2, \rr(x)])+S([S(r^2),x])\o r^1+[S(r^2),\rr(x)]\o r^1\\
 &&-\lambda S([S(r^2), \rr(x)]\o r^1+S([x, S(r^2)])\o r^1+[\rr(x),S(r^2)]\o r^1\\
 &&-\lambda S([\rr(x),S(r^2)])\o r^1,\\
 \lambda (S\o \rr)\delta(\rr(x))\hspace{-3mm}&=&\hspace{-3mm}-\lambda S(r^1)\o \rr([r^2, \rr(x)])+\lambda S([r^2, \rr(x)])\o \rr(r^1)\\
 &&+\lambda S([\rr(x), r^2])\o \rr(r^1),
 \end{eqnarray*}
 then Eq.(\ref{eq:88}) $\Longleftrightarrow$  Eq.(\ref{eq:106}).

 \ref{it:107} For all $x\in \gg$, by Eq.(\ref{eq:94}) we have
 \begin{eqnarray*}
 (\rr\o \id)\delta(\rr(x))\hspace{-3mm}&=&\hspace{-3mm}-\rr(r^1)\o [r^2, \rr(x)]+\rr ([r^2, \rr(x)])\o r^1+\rr([\rr(x), r^2])\o r^1\\
 &\stackrel {(\ref{eq:2})}{=}&-\rr(r^1)\o [r^2, \rr(x)]+[\rr(r^2), \rr(x)]\o r^1-\rr([\rr (r^2), x])\o r^1\\
 &&+\lambda \rr([\rr(r^2), \rr(x)])\o r^1+[\rr(x),\rr(r^2)]\o r^1-\rr([x, \rr(r^2)])\o r^1\\
 &&+\lambda \rr([\rr(x), \rr(r^2)])\o r^1,\\
 (\rr\o S)\delta(x)\hspace{-3mm}&=&\hspace{-3mm}-\rr(r^1)\o S([r^2, x])+\rr([r^2, x])\o S(r^1)+\rr([x, r^2])\o S(r^1),\\
 (\id \o S)\delta(\rr(x))\hspace{-3mm}&=&\hspace{-3mm}-r^1\o S([r^2, \rr(x)])+[r^2,\rr (x)]\o S(r^1)+[\rr(x),r^2]\o S(r^1)\\
 &\stackrel {(\ref{eq:25})}{=}&-r^1\o S([S(r^2),x])-r^1\o [S(r^2),\rr(x)]+r^1\o \lambda S([S(r^2),\rr(x)])\\
 &&+[r^2,\rr(x)]\o S(r^1)+[\rr(x),r^2]\o S(r^1),\\
 \lambda (\rr\o S)\delta(\rr(x))\hspace{-3mm}&=&\hspace{-3mm}-\lambda \rr(r^1)\o S([r^2, \rr(x)])+\lambda \rr([r^2, \rr(x)])\o S(r^1)\\
 &&+\lambda \rr([\rr(x), r^2])\o S(r^1),
 \end{eqnarray*}
 then Eq.(\ref{eq:89}) $\Longleftrightarrow$ Eq.(\ref{eq:108}). We finish the proof.
 \end{proof}

 \begin{lem}\label{lem:112}
 \begin{enumerate}[(1)]
 \item \label{it:113} If $(S\o \id-\id \o \rr)(r)=0$, then $(\id\o S-\rr\o \id)(r^{\tau})=0$;
 \end{enumerate}
 \begin{enumerate}[(2)]
 \item \label{it:114} If $(\id \o S-\rr\o \id)(r)=0$, then $(S \o \id-\id\o \rr )(r^{\tau})=0$;
 \end{enumerate}
 \begin{enumerate}[(3)]
 \item \mlabel{it:115} If $r=r^{\tau}$, then $(S\o \id-\id\o \rr)(r)=0$ $\Longleftrightarrow$ $(\rr\o \id-\id\o S)(r)=0$.
 \end{enumerate}
 \end{lem}

 \begin{proof}
 Straightforward.
 \end{proof}

 \begin{thm}\label{thm:116}  Let $(\gg, \rr)$ be an $S$-adjoint admissible Reynolds Leibniz algebra. Then the linear map $\delta$ defined by Eq.(\ref{eq:94}) induces a $\rr ^*$-adjoint admissible Reynolds Leibniz algebra $((\gg^*, \delta^*),S^*)$ such that $\big((\gg, \rr), \delta, S\big)$ is a Reynolds Leibniz bialgebra if and only if Eqs.(\ref{eq:95}), (\ref{eq:96}), (\ref{eq:97}), (\ref{eq:104}), (\ref{eq:106}) and (\ref{eq:108}) hold.
 \end{thm}

 \subsection{Admissible cLYBe in a Reynolds Leibniz algebra}

 \begin{cor}\label{cor:117}   Let $(\gg, \rr)$ be an $S$-adjoint admissible Reynolds Leibniz algebra and $r\in \gg\o \gg$.  If Eqs.(\ref{eq:95}),(\ref{eq:96}), (\ref{eq:97}) and the following equations hold,
 \begin{eqnarray}
 (S\o \id-\id\o \rr)(r)=0,\label{eq:118}\\
 (\id\o S-\rr\o \id)(r)=0,\label{eq:119}
 \end{eqnarray}
 then $\big((\gg, \rr), \delta, S\big)$ is a Reynolds Leibniz bialgebra, where $\delta$ is defined by Eq.(\ref{eq:94}).
 \end{cor}

 \begin{defi}\label{defi:122}  Let $(\gg, \rr)$ be a Reynolds Leibniz algebra, $r\in \gg\o\gg$, $S:\gg\lr \gg$ a linear map. Then Eq.(\ref{eq:100}) together with Eqs.(\ref{eq:118}) and (\ref{eq:119}) is called an \textbf{$S$-admissible classical Leibniz Yang-Baxter equation in $(\gg, \rr)$} or simply an \textbf{$S$-admissible cLYBe in $(\gg, \rr)$}.
 \end{defi}

 \begin{rmk}\label{rmk:123} Eq.(\ref{eq:100}) is simply the cLYBe in Leibniz algebra (see\cite{LMW}), as an analogue of classical Yang-Baxter equation in a Lie algebra. Also by Lemma \ref{lem:112} , if $r$ is symmetric (that is, $r^{\tau}=r$), then Eq.(\ref{eq:118}) holds if and only if Eq.(\ref{eq:119}) holds.
 \end{rmk}

 Then by Corollary \ref{cor:117}, we have

 \begin{cor} \label{cor:124}  Let $(\gg, \rr)$ be an $S$-adjoint admissible Reynolds Leibniz algebra, $r\in \gg\o \gg$ be a symmetric solution of the $S$-admissible cLYBe in $(\gg, \rr)$. Then $\big((\gg, \rr), \delta_r, S\big)$ is a Reynolds Leibniz bialgebra, where $\delta_r$ is defined by Eq.(\ref{eq:94}). In this case we call this Reynolds Leibniz bialgebra \textbf{triangular} and denote it by $((\gg, \rr),r,S)$. 
 \end{cor}

 For a vector space $\gg$, by the isomorphism $\gg\o \gg\cong Hom(\gg^*, K)\otimes\gg\cong Hom(\gg^*, \gg)$  and $r\in \gg\o \gg$, we define
 \begin{eqnarray}\label{eq:125}
 r^{\sharp}: \gg^*\lr \gg,~~r^{\sharp}(\xi)=\langle\xi, r^1\rangle r^2.
 \end{eqnarray}

 We call $r\in \gg\o \gg$ \textbf{nondegenerate} if the map $r^{\sharp}: \gg^*\lr \gg$ defined by Eq.(\ref{eq:125}) is bijective.

 \begin{thm}\label{thm:126}  Let $(\gg, \rr)$ be a Reynolds Leibniz algebra, $r$ be a symmetric element in $\gg\o \gg$ and $S:\gg\lr \gg$ be a linear map. Then $r$ is a solution of the $S$-admissible cLYBe in $(\gg, \rr)$ if and only if $r^{\sharp}$ satisfies the following conditions:
 \begin{eqnarray*}
 &[r^{\sharp}(\xi),r^{\sharp}(\eta)]=r^{\sharp}\Big(L^*\big(r^{\sharp}(\xi)\big)\eta
 +(-L^*-R^*)\big(r^{\sharp}(\eta)\big)\xi\Big),\forall~\xi,\eta\in\gg^*&\label{eq:127}\\
 &\rr r^{\sharp}=r^{\sharp}S^*.&\label{eq:128}
 \end{eqnarray*}
 \end{thm}

 \begin{proof}
 It is straightforward by \cite[Proposition 4.8]{TS}.
 \end{proof}

 \subsection{$\mathcal{O}$-operator on a Reynolds Leibniz algebra} Next we extend the properties of $r^{\sharp}$ in Theorem \ref{thm:126} to a general setting. 

 \begin{defi}\label{defi:130} Let $(\gg, \rr)$ be a Reynolds Leibniz algebra, $(V, \rho^L, \rho^R)$ be a representation of $(\gg, [,])$ and $\a:V\lr V$ be a linear map. A linear map $T:V\lr \gg$ is called a \textbf{weak $\mathcal{O}$-operator associated to $(V, \rho^L, \rho^R)$ and $\a$} if $T$ satisfies
 \begin{eqnarray}
 &[T(u),T(v)]=T\Big(\rho^L(T(u))v+\rho^R(T(v))u\Big),\forall~u, v\in V, &\label{eq:131}\\
 &\rr T=T\alpha.&\label{eq:132}
 \end{eqnarray}
 in addition, if $(V, \rho^L, \rho^R, \a)$ is a representation of $(\gg, \rr)$, then we call $T$ an \textbf{$\mathcal{O}$-operator associated to $(V, \rho^L, \rho^R, \a)$}.
 \end{defi}

 \begin{ex}\label{ex:133} Let $(\gg, \rr)$ be a Reynolds Leibniz algebra. Then the identity map $\id$ on $\gg$ is an $\mathcal{O}$-operator associated to $(\gg, L, 0, \rr)$ or $(\gg, 0, R, \rr)$.
 \end{ex}

 Theorem \ref{thm:126} can be rewritten as follows via an $\mathcal{O}$-operator. 

 \begin{cor}\label{cor:134}  Let $(\gg, \rr)$ be a Reynolds Leibniz algebra, $r\in \gg\o \gg$ be symmetric, $S:\gg\lr \gg$ be a linear map. Then $r$ is a solution of the $S$-admissible cLYBe in $(\gg, \rr)$ if and only if $r^{\sharp}$ is a weak $\mathcal{O}$-operator associated to $(\gg^*, L^*, -L^*-R^*)$ and $S^*$. In addition, if $(\gg, \rr)$ is an $S$-adjoint admissible Reynolds Leibniz algebra, then $r$ is a solution of the $S$-admissible cLYBe in $(\gg, \rr)$ if and only if $r^{\sharp}$ is an $\mathcal{O}$-operator associated to the representation $(\gg^*, L^*, -L^*-R^*, S^*)$.
 \end{cor}

 In what follows, we will prove that $\mathcal{O}$-operator can provide solutions of the $S$-admissible cLYBe in semi-direct product Reynolds Leibniz algebras and then produce Reynolds Leibniz bialgebras.

 \begin{thm}\label{thm:135}  Let $(\gg,\rr)$ be a Reynolds Leibniz algebra, $(V, \rho^L, \rho^R)$ be a representation of $(\gg, [,])$, $S :\gg\lr \gg$  and $\a, \beta:V\lr V$ be linear maps. Then the following conditions are equivalent:
 \begin{enumerate}[(1)]
 \item \label{it:136} There is a Reynolds Leibniz algebra $(\gg\ltimes_{\rho^L,\; \rho^R} V, \rr+\a)$ such that the linear map $S+\beta$ on $\gg\oplus V$ is adjoint admissible to $(\gg\ltimes_{\rho^L,\; \rho^R} V, \rr+\a)$.
 \item \label{it:137} There is a Reynolds Leibniz algebra $(\gg\ltimes_{\rho^{L*}, -\rho^{L*}-\rho^{R*}} V^*, \rr+\beta^*)$ such that the linear map $S+\a^*$ on $\gg\oplus V^*$ is  adjoint admissible to $(\gg\ltimes_{\rho^{L*}, -\rho^{L*}-\rho^{R*}} V^*, \rr+\beta^*)$.
 \item \label{it:138} The following conditions are satisfied:
 \begin{enumerate}[(a)]
 \item \label{it:139} $(V, \rho^L, \rho^R, \a)$ is a representation of $(\gg, \rr)$;
 \item \label{it:140} $S$ is adjoint admissible to $(\gg, \rr)$;
 \item \label{it:141} $\beta$ is admissible to $(\gg, \rr)$ associated to $(V, \rho^L, \rho^R)$;
 \item \label{it:142} For all $x\in \gg$ and $v\in V$, the following equations hold:
 \begin{eqnarray}
 &&\beta(\rho^L(S(x))v)+\rho^L(S(x))\alpha (v)=\beta(\rho^L(x)\alpha(v))
 +\lambda \beta(\rho^L(S(x))\alpha(v)),\label{eq:143} \\
 &&\beta(\rho^R(S(x))v)+\rho^R(S(x))\alpha(v)=\beta(\rho^R(x)\alpha(v))
 +\lambda \beta(\rho^R(S(x))\alpha(v)).\label{eq:144}
 \end{eqnarray}
 \end{enumerate}
 \end{enumerate}
 \end{thm}

 \begin{proof} $\ref{it:136}\Leftrightarrow\ref{it:138}$: According to Proposition \ref{pro:12}, $(\gg\ltimes_{\rho^L,\; \rho^R} V, \rr+\alpha)$ is a Reynolds Leibniz algebra if and only if $(V, \rho{^L}, \rho{^R}, \alpha)$ is a representation of $(\gg, \rr)$. To express the condition that $S+\beta$ on $\gg\oplus V$ is adjoint admissible to $(\gg\ltimes_{\rho^L, \rho^R} V, \rr+\alpha)$, we perform the following calculations. For all $x, y\in \gg, u, v\in V$, we have
 \begin{align*}
 (S+\beta)[(x+u), (S+\beta)(y+v)]_{\ltimes} &= S([x, S(y)])+\beta(\rho^L(x)\beta(v))+\beta(\rho^R(S(y))u), \\
 [(\rr+\alpha)(x+u), (S+\beta)(y+v)]_{\ltimes} &= [\rr(x), S(y)]+\rho^L(\rr (x))\beta(v)+\rho^R(S(y))\alpha(u), \\
 (S+\beta)[(\rr+\alpha)(x+u), y+v]_{\ltimes} &= S([\rr(x), y])+\beta(\rho^L(\rr (x))v)+\beta(\rho^R(y)\alpha(u)), \\
 \lambda(S+\beta)[(\rr+\alpha)(x+u), (S+\beta)(y+v)]_{\ltimes} &=\lambda S([\rr(x), S(y)])+\lambda \beta(\rho^L(\rr(x))\beta(v))\\
 &+\lambda \beta(\rho^R(S(y))\alpha(u)), \\
 (S+\beta)[(S+\beta)(x+u), y+v]_{\ltimes} &= S([S(x), y])+\beta(\rho^L(S(x))v)+\beta(\rho^R(y)\beta(u)), \\
 [(S+\beta)(x+u), (\rr+\alpha)(y+v)]_{\ltimes} &= [S(x), \rr(y)]+\rho^L(S(x))\alpha (v)+\rho^R(\rr(y))\beta(u), \\
 (S+\beta)[x+u, (\rr+\alpha)(y+v)]_{\ltimes} &= S([x, \rr (y)])+\beta(\rho^L(x))\alpha(v))+\beta(\rho^R(\rr(y))u), \\
 \lambda(S+\beta)[(S+\beta)(x+u), (\rr+\alpha)(y+v)]_{\ltimes} &= \lambda S([S(x), \rr (y)])+\lambda \beta(\rho^L(S(x))\alpha(v))\\
 &\hspace{6mm}+\lambda \beta(\rho^R(\rr(y))\beta(u)).
 \end{align*}
 Hence Eq.(\ref{eq:24}) holds (where $\rr$ replaced by $\rr+\a$, $S$ replaced by $S+\beta$, $x$ replaced by $x+u$, $y$ replaced by $y+v$) if and only if Eq.(\ref{eq:24})(corresponding to $u=v=0$), Eq.(\ref{eq:19})(corresponding to $y=u=0$) and Eq.(\ref{eq:144})(corresponding to $x=v=0$) hold, where $x$ replaced by $y$, $v$ replaced by $u$ (corresponding to $y=v=0$), Eq.(\ref{eq:25}) holds (where $\rr$ replaced by $\rr+\a$, $S$ replaced by $S+\beta$, $x$ replaced by $x+u$, $y$ replaced by $y+v$) if and only if Eq.(\ref{eq:25})(corresponding to $u=v=0$), Eq.(\ref{eq:20})(where $x$ replaced by $y$, $v$ replaced by $u$ (corresponding to $x=v=0$)) and (\ref{eq:143}) holds (corresponding to $y=u=0$). Hence, Item \ref{it:136} holds if and only if Item \ref{it:138} holds.

 $\ref{it:137}\Leftrightarrow\ref{it:138}$: Based on $\ref{it:136}\Leftrightarrow\ref{it:138}$, let $V=V^*, \rho^L=\rho^{L*}, \rho^R=-\rho^{L*}-\rho^{R*},\beta=\a^*, \a=\beta^*$, then we have \ref{it:137} holds if and only if
 \begin{enumerate}[(i)]
 \item \label{it:145} $\big(V^*, \rho^{L*}, -\rho^{L*}-\rho^{R*}, \beta^*\big)$ is a representation of $(\gg, \rr)$, that is, \ref{it:138}\ref{it:141} holds;
 \item \label{it:146} $S$ is adjoint admissible to $(\gg, \rr)$ , that is, \ref{it:138}\ref{it:140} holds;
 \item \label{it:147} $\a^*$ is admissible to $(\gg, \rr)$ associated to $\big(V^*,  \rho^{L*}, -\rho^{L*}-\rho^{R*}\big)$, that is, \ref{it:138}\ref{it:139} holds;
 \item \label{it:148} For all $x\in \gg$ and $v\in V$, the following equations hold:
 \begin{eqnarray}
 &\alpha^*\rho^{L*}(S(x))v^*+\rho^{L*}(S(x))\beta^*(v^*)
 =\alpha^*\rho^{L*}(x)\beta^*(v^*))+\lambda \alpha^*\rho^{L*}(S(x))\beta^*(v^*)),\label{eq:190} \\
 &\alpha^*(-\rho^{L*}-\rho^{R*})(S(x))v^*)+(-\rho^{L*}-\rho^{R*})(S(x))\beta^*(v^*))\nonumber\\
 &\hspace{22mm}=\alpha^*(-\rho^{L*}-\rho^{R*})(x)\beta^*(v^*)+\lambda \alpha^*((-\rho^{L*}-\rho^{R*})(S(x))\beta^*(v^*)).\label{eq:149}
 \end{eqnarray}
 \end{enumerate}

 \noindent Therefore according to the Eqs.(\ref{eq:16}) and (\ref{eq:17}), we can obtain Eq.(\ref{eq:190}) holds if and only if Eq.(\ref{eq:143}) holds, Eq.(\ref{eq:149}) holds if and only if Eq.(\ref{eq:144}) holds. Hence, Item \ref{it:137} holds if and only if Item \ref{it:138} holds.
 \end{proof}

 \begin{thm}\label{thm:150}  Let $(\gg,\rr)$ be $\beta$-admissible on $(V, \rho{^L}, \rho{^R})$, $S:\gg\lr \gg$, $\a:V\lr V$ be two linear maps. Let $T:V\lr \gg$ be a linear map, $r=T+\tau(T)$ be an element in $(\gg\ltimes_{\rho^{L*}, -\rho^{L*}-\rho^{R*}} V^*)\o (\gg\ltimes_{\rho^{L*}, -\rho^{L*}-\rho^{R*}} V^*)$. 
 \begin{enumerate}[(1)]
 \item \label{it:151} The element $r=T+\tau(T)$ is a symmetric solution of the $(S+\a^*)$-admissible cLYBe in the Reynolds Leibniz algebra $(\gg\ltimes_{\rho^{L*}, -\rho^{L*}-\rho^{R*}}V^*, \rr+\beta^*)$  if and only if $T$ is a weak $\mathcal{O}$-operator associated to $(V, \rho{^L}, \rho{^R})$ and $\alpha$, and satisfies $T\circ \beta=S\circ T$. 
 \item \label{it:152} Assume that $(V, \rho{^L}, \rho{^R}, \a)$ is a representation of $(\gg, \rr)$. If $T$ is an $\mathcal{O}$-operator associated to $(V, \rho{^L}, \rho{^R}, \a)$ and $T\circ \beta=S\circ T$, then $r=T+\tau(T)$ is a symmetric solution of the $(S+\a^*)$-admissible cLYBe in the Reynolds Leibniz algebra $(\gg\ltimes_{\rho^{L*}, -\rho^{L*}-\rho^{R*}}V^*, \rr+\beta^*)$. In addition, if $(\gg,\rr)$ is $S$-adjoint admissible and Eqs.(\ref{eq:143})-(\ref{eq:144}) hold, then $(\gg\ltimes_{\rho^{L*}, -\rho^{L*}-\rho^{R*}}V^*, \rr+\beta^*)$ is $(S+\a^*)$-adjoint admissible. In this case, there is a Reynolds Leibniz bialgebra $((\gg\ltimes_{\rho^{L*}, -\rho^{L*}-\rho^{R*}}V^*, \rr+\beta^*), \delta, S+\a^*)$, where $\delta=\delta_r$ is defined by Eq.(\ref{eq:94}) and $r=T+\tau(T)$.
 \end{enumerate}
 \end{thm}

 \begin{proof}
 \ref{it:151} Let $\{e_1,e_2,...,e_n\}$ be a basis of $V$ and $\{e^1,e^2,...,e^n\}$ be its dual basis. Then $r=T+\tau(T)$ corresponds to $\sum_{i=1}^n T(e_i)\o e^i+e^i\o T(e_i) \in(\gg\ltimes_{\rho^{L*},-\rho^{L*}-\rho^{R*}}V^*)\o (\gg\ltimes_{\rho^{L*},-\rho^{L*}-\rho^{R*}}V^*)$.
 By \cite[Theorem 4.15]{MSZ}, we have
 \begin{eqnarray*}
 r_{12}r_{23}+r_{13}r_{23}=r_{12}^{\tau}r_{13}+r_{13}r_{12}^{\tau}\Longleftrightarrow Eq.{(\ref{eq:131})}.
 \end{eqnarray*}
 Note that
 \begin{eqnarray*}
 ((\rr+\beta^*)\o \id)(r)&=&\sum_{i=1}^n\rr T(e_i)\o e^i+\beta^*(e^i)\o T(e_i), \\
 (\id\o (S+\alpha^*))(r)&=&\sum_{i=1}^nT(e_i)\o\alpha^*(e^i)+e^i\o ST(e_i).
 \end{eqnarray*}
 Further
 \begin{eqnarray*}
 \sum_{i=1}^n\beta^*(e^i)\o T(e_i)&=&\sum_{i=1}^n\sum_{j=1}^n\langle\beta^*(e^i), e_j\rangle e^j\o T(e_i)=\sum_{j=1}^n e^j\o \sum_{i=1}^n\langle e^i, \beta(e_j)\rangle T(e_i)\\
 &=&\sum_{i=1}^n e^i\o \sum_{j=1}^n\langle e^j, \beta(e_i)\rangle T(e_j)=\sum_{i=1}^ne^i\o T\beta(e_i),\\
 \sum_{i=1}^nT(e_i)\o\alpha^*(e^i)&=&\sum_{i=1}^n\sum_{j=1}^n T(e_i)\o \langle\alpha^*(e^i), e_j\rangle e^j=\sum_{i=1}^n T(e_i)\langle e^i,\alpha( e_j)\rangle\o \sum_{j=1}^ne^j\\
 &=&\sum_{i=1}^nT\alpha(e_i)\o e^i.
 \end{eqnarray*} 
 Hence $((\rr+\beta^*)\o \id)(r)=(\id\o (S+\alpha^*))(r)$ if and only if $T\beta=S T$ and $\rr T=T\alpha$. Therefore, the conclusion follows.

 \ref{it:152} It follows from \ref{it:151} and Theorem \ref{thm:135}.
 \end{proof}

 We now consider some special spaces: $\beta=\pm\alpha$ or $-\alpha+\theta \id$ or $\beta=\theta \alpha^{-1}$. For these cases, they all satisfy that the double dual of a representation is itself. For simplicity, we introduce a Laurent polynomial $\Pi\in K[x,x^{-1}]$ such that $\Pi(x)$ is either $\pm x$, or $-x+\theta$, or $\theta x^{-1}$ when $x$ is invertible and $0\neq \theta\in K$. Then the special cases above can be denoted by $\Pi(\alpha)$. 

 To emphasize, for all $\Pi$ in the set
 \begin{eqnarray*}
 \{\pm x\}\cup(-x+K^{\times})\cup K^{\times}x^{-1}, K^{\times}:=K\backslash\{0\}.
 \end{eqnarray*}
 We have $\Pi^2(\alpha)=\alpha$ and $\Pi(\alpha^{*})=\Pi(\alpha)^{*}$. Moreover, for any linear map $T: V\rightarrow \gg$, it is obvious that $T\Pi(\alpha)=\Pi(\rr)T$ when $T\alpha=\rr T$.

 Applying Theorem \ref{thm:135}, we have
 \begin{pro}\label{pro:202}
 Let ($\gg,\rr$) be a Reynolds Leibniz algebra, ($V,\rho^{L},\rho^{R}$) be a representation of $(\gg,[,]),\alpha:V\rightarrow V$ be a linear map. For $\Pi\in {\pm x}\cup(-x+K^{\times})\cup K^{\times}x^{-1}$, there is a Reynolds Leibniz algebra $(\gg\ltimes_{\rho^{L*}, -\rho^{L*}-\rho^{R*}} V^*, \rr+\Pi(\alpha)^*)$ such that ($\Pi(\rr)+\a^*$)-adjoint admissible if and only if the $\Pi$-admissible equations (associated to the quadruple ($V,\rho^{L},\rho^{R},\alpha$)) hold. Here
 \begin{enumerate}[(1)]
 \item \label{it:204} When $\Pi=x$, that is, $\beta=\a$, $S=\rr$, the $\Pi$-admissible equations are: $(V, \rho^L, \rho^R, \a)$ is a representation of $(\gg, \rr)$ and
 \begin{eqnarray}
 &\rr([x, \rr(y)])=\rr([\rr(x), y])=\lambda\rr([\rr(x), \rr(y)]),&\label{eq:207}\\
 &\alpha(\rho^L(\rr(x))v)=\alpha(\rho^L(x)\alpha(v))=\lambda\alpha(\rho^L(\rr (x))\alpha(v)),&\label{eq:208}\\
 &\alpha(\rho^R(\rr(x))v)=\alpha(\rho^R(x)\alpha(v))=\lambda\alpha(\rho^R(\rr (x))\alpha(v)),&\label{eq:209}
 \end{eqnarray}
 where $\forall x, y\in \gg, v\in V$.
 \item \label{it:213} When $\Pi=-x$, that is, $\beta=-\a$, $S=-\rr$ , the $\Pi$-admissible equations turn to be that $(V, \rho^L, \rho^R, \a)$ is a representation of $(\gg, \rr)$.
 \item \label{it:214} When $\Pi=-x+\theta$ with $\theta\neq0$, that is $\beta=-\a+\theta\id_V$, $S=-\rr+\theta\id_\gg$, the $\Pi$-admissible equations are: $(V, \rho^L, \rho^R, \a)$ is a representation of $(\gg, \rr)$ and
 \begin{eqnarray}
 &\theta [x, y] -[x,\rr(y)]-\rr([x,y])=\lambda \Big(\theta [\rr(x), y] - [\rr(x), \rr(y)] -\rr([\rr(x),y])\Big),&\label{eq:215}\\
 &\theta [x, y] -[\rr(x),y]-\rr([x,y])=\lambda \Big(\theta [x, \rr(y)] - [\rr(x), \rr(y)] -\rr([x,\rr(y)])\Big),&\label{eq:216}\\
 &\theta\rho^L(x)v-\rho^L(x)\alpha(v)-\alpha(\rho^L(x)v)=\lambda\Big(\theta\rho^L(\rr (x))v-\rho^L(\rr(x))\alpha(v)-\alpha(\rho^L(\rr(x))v)\Big),&\label{eq:217}\\
 &\theta\rho^R(x)v-\rho^R(x)\alpha(v)-\alpha(\rho^R(x)v)=\lambda\Big(\theta\rho^R(\rr (x))v-\rho^R(\rr(x))\alpha(v)-\alpha(\rho^R(\rr(x))v)\Big),&\label{eq:218}\\
 &\theta\rho^L(x)v-\rho^L(\rr (x))v-\alpha(\rho^L(x)v)=\lambda\Big(\theta\rho^L(x)\alpha(v)-\rho^L(\rr (x))\alpha(v)-\alpha(\rho^L(x)\alpha(v))\Big),&\label{eq:219}\\
 &\theta\rho^R(x)v-\rho^R(\rr (x))v-\alpha(\rho^R(x)v)=\lambda\Big(\theta\rho^R(x)\alpha(v)-\rho^R(\rr (x))\alpha(v)-\alpha(\rho^R(x)\alpha(v))\Big),&\label{eq:220}
 \end{eqnarray}
 where $\forall x, y\in \gg, v\in V$.
 \item \label{it:221} When $\Pi=\theta x^{-1}$, $\theta\neq 0$ (in which case assume $\rr$ and $\a$ are invertible), that is, $\beta=\theta\a^{-1}$, $S=\theta \rr^{-1}$, the $\Pi$-admissible equations are: $(V, \rho^L, \rho^R, \a)$ is a representation of $(\gg, \rr)$ and
 \begin{eqnarray}
 &\theta [x, y]-\rr([x, \rr(y)])=\lambda\Big(\theta [\rr(x),y]-\rr([\rr(x), \rr (y)])\Big),&\label{eq:222}\\
 &\theta [x, y]-\rr([\rr(x), y])=\lambda\Big(\theta [x,\rr(y)]-\rr([\rr(x), \rr (y)])\Big),&\label{eq:223}\\
 &\theta \rho^L(x)v -\alpha(\rho^L(x)\alpha(v))=\lambda\Big(\theta \rho^L(\rr (x))v -\alpha(\rho^L(\rr(x))\alpha(v))\Big),&\label{eq:224}\\
 &\theta \rho^R(x)v -\alpha(\rho^R(x)\alpha(v))=\lambda\Big(\theta \rho^R(\rr (x))v -\alpha(\rho^R(\rr(x))\alpha(v))\Big),&\label{eq:225}\\
 &\theta \rho^L(x)v -\alpha(\rho^L(\rr (x))v)=\lambda\Big(\theta \rho^L(x)\alpha(v)-\alpha(\rho^L(\rr (x))\alpha(v))\Big),&\label{eq:226}\\
 &\theta \rho^R(x)v -\alpha(\rho^R(\rr (x))v)=\lambda\Big(\theta \rho^R(x)\alpha(v)-\alpha(\rho^R(\rr (x))\alpha(v))\Big),&\label{eq:227}
 \end{eqnarray}
 where $\forall x, y\in \gg, v\in V$.
 \end{enumerate}
 \end{pro}

 \begin{proof}
 The statement is a consequence of Theorem \ref{thm:135} in the case that $\beta=\Pi(\alpha)$ and $S=\Pi(\rr)$. For the detailed proof the reader can refer to \cite[Proposition 4.22]{BGM}. 
 \end{proof}

 By Proposition \ref{pro:202} and Theorem \ref{thm:de:87}, we have
 \begin{cor}\label{cor:pro:202} 
 Let ($\gg, \rr$) be a Rota-Baxter Leibniz algebra of weight 0, ($V,\rho^{L},\rho^{R}$) be a representation of $(\gg,[,])$, $\alpha:V\rightarrow V$ be a linear map. For $\Pi\in {\pm x}\cup(-x+K^{\times})\cup K^{\times}x^{-1}$, there is a Rota-Baxter Leibniz algebra $(\gg\ltimes_{\rho^{L*}, -\rho^{L*}-\rho^{R*}} V^*, \rr+\Pi(\alpha)^*)$ of weight 0 such that ($\Pi(\rr)+\a^*$)-adjoint admissible if and only if the $\Pi$-admissible equations (associated to the quadruple ($V,\rho^{L},\rho^{R},\alpha$)) hold. Here
 \begin{enumerate}[(1)]
 \item \label{it:204} When $\Pi=x$, that is, $\beta=\a$, $S=\rr$, the $\Pi$-admissible equations are: $(V, \rho^L, \rho^R, \a)$ is a representation of $(\gg, \rr)$ and
 \begin{eqnarray}
 &\rr([x, \rr(y)])=\rr([\rr(x), y])=0,&\label{eq:207-1}\\
 &\alpha(\rho^L(\rr(x))v)=\alpha(\rho^L(x)\alpha(v))=0,&\label{eq:208-1}\\
 &\alpha(\rho^R(\rr(x))v)=\alpha(\rho^R(x)\alpha(v))=0,&\label{eq:209-1}
 \end{eqnarray}
 where $\forall x, y\in \gg, v\in V$.
 \item \label{it:213} When $\Pi=-x$, that is, $\beta=-\a$, $S=-\rr$ , the $\Pi$-admissible equations turn to be that $(V, \rho^L, \rho^R, \a)$ is a representation of $(\gg, \rr)$.
 \item \label{it:214} When $\Pi=-x+\theta$ with $\theta\neq0$, that is $\beta=-\a+\theta\id_V$, $S=-\rr+\theta\id_\gg$, the $\Pi$-admissible equations are: $(V, \rho^L, \rho^R, \a)$ is a representation of $(\gg, \rr)$ and
 \begin{eqnarray}
 &\theta [x, y]=[x,\rr(y)]+\rr([x,y]),&\label{eq:215-1}\\
 &\theta [x, y]=[\rr(x),y]+\rr([x,y]),&\label{eq:216-1}\\
 &\theta\rho^L(x)v=\rho^L(x)\alpha(v)+\alpha(\rho^L(x)v),&\label{eq:217-1}\\
 &\theta\rho^R(x)v=\rho^R(x)\alpha(v)+\alpha(\rho^R(x)v),&\label{eq:218-1}\\
 &\theta\rho^L(x)v=\rho^L(\rr (x))v+\alpha(\rho^L(x)v),&\label{eq:219-1}\\
 &\theta\rho^R(x)v=\rho^R(\rr (x))v+\alpha(\rho^R(x)v),&\label{eq:220-1}
 \end{eqnarray}
 where $\forall x, y\in \gg, v\in V$.
 \item \label{it:221} When $\Pi=\theta x^{-1}$, $\theta\neq 0$ (in which case assume $\rr$ and $\a$ are invertible), that is, $\beta=\theta\a^{-1}$, $S=\theta \rr^{-1}$, the $\Pi$-admissible equations are: $(V, \rho^L, \rho^R, \a)$ is a representation of $(\gg, \rr)$ and
 \begin{eqnarray}
 &\theta [x, y]=\rr([x, \rr(y)]),&\label{eq:222-1}\\
 &\theta [x, y]=\rr([\rr(x), y]),&\label{eq:223-1}\\
 &\theta \rho^L(x)v=\alpha(\rho^L(x)\alpha(v)),&\label{eq:224-1}\\
 &\theta \rho^R(x)v=\alpha(\rho^R(x)\alpha(v)),&\label{eq:225-1}\\
 &\theta \rho^L(x)v=\alpha(\rho^L(\rr (x))v),&\label{eq:226-1}\\
 &\theta \rho^R(x)v=\alpha(\rho^R(\rr (x))v),&\label{eq:227-1}
 \end{eqnarray}
 where $\forall x, y\in \gg, v\in V$.
 \end{enumerate}
 \end{cor}

 \begin{thm}\label{thm:228}  Let $(\gg, \rr)$ be a Reynolds Leibniz algebra, $(V, \rho{^L}, \rho{^R})$ be a representation of $(\gg, [,])$ and $\a:V\lr V, T:V\lr \gg$ be linear maps. Let $\Pi\in {\pm x}\cup(-x+K^{\times})\cup K^{\times}x^{-1}$.
 \begin{enumerate}[(1)]
 \item \label{it:229}Let $(V, \rho{^L}, \rho{^R}, \Pi(\a))$ be an admissible quadruple of $(\gg, \rr)$. Then $r=T+\tau(T)$ is a symmetric solution of the ($\Pi(\rr)+\a^*$)-admissible cLYBe in the Reynolds Leibniz algebra $(\gg\ltimes_{\rho^{L*}, -\rho^{L*}-\rho^{R*}}V^*,$ $\rr+\Pi(\a^*))$ if and only if $T$ is a weak $\mathcal{O}$-operator associated to $(V, \rho{^L}, \rho{^R})$ and $\a$.
 \item \label{it:230}Assume the validity of the $\Pi$-admissible equations, given respectively by Eqs.(\ref{eq:3})-(\ref{eq:4}), (\ref{eq:207})-(\ref{eq:209}) for $\Pi=+ x$, by Eqs.(\ref{eq:3})-(\ref{eq:4}) for $\Pi=- x$, Eqs.(\ref{eq:3})-(\ref{eq:4}), (\ref{eq:215})-(\ref{eq:220}) for $\Pi\in-x+K^{\times}$ , Eqs.(\ref{eq:3})-(\ref{eq:4}), (\ref{eq:222})-(\ref{eq:227}) for $\Pi\in K^{\times}x^{-1}$, then $(V, \rho{^L}, \rho{^R},$ $ \a)$ is a representation of $(\gg, \rr)$ and there is a $(\Pi(\rr)+\a^*)$-adjoint admissible Reynolds Leibniz algebra $(\gg\ltimes_{\rho^{L*}, -\rho^{L*}-\rho^{R*}}V^*,$ $\rr+\Pi(\a^*))$. 
    If $T$ is an $\mathcal{O}$-operator associated to $(V, \rho{^L}, \rho{^R}, \a)$, then $r=T+\tau(T)$ is a symmetric solution of the ($\rr+\Pi(\a^*))$-admissible cLYBe in the Reynolds Leibniz algebra $(\gg\ltimes_{\rho^{L*}, -\rho^{L*}-\rho^{R*}}V^*,$ $\rr+\Pi(\a^*))$. Furthermore, there is a Reynolds Leibniz bialgebra $(\gg\ltimes_{\rho^{L*}, -\rho^{L*}-\rho^{R*}}V^*,\rr+\Pi(\a^*)), \delta, \Pi(\rr)+\a^*)$, where the linear map $\delta=\delta_r$ is defined by Eq.(\ref{eq:94}) and $r=T+\tau(T)$. 
 \end{enumerate}
 \end{thm}

 \begin{proof}
 \ref{it:229} From the Theorem \ref{thm:150}\ref{it:151}.

 \ref{it:230} From the Proposition \ref{pro:202} and Theorem \ref{thm:150}\ref{it:152}.
 \end{proof}

 By Theorem \ref{thm:228} and Theorem \ref{thm:de:87}, we have
 \begin{cor}\label{cor:thm:228} Let $(\gg, \rr)$ be a Rota-Baxter Leibniz algebra of weight 0, $(V, \rho{^L}, \rho{^R})$ be a representation of $(\gg, [,])$ and $\a:V\lr V, T:V\lr \gg$ be linear maps. Let $\Pi\in {\pm x}\cup(-x+K^{\times})\cup K^{\times}x^{-1}$.
 \begin{enumerate}[(1)]
 \item \label{it:229} Let $(V, \rho{^L}, \rho{^R}, \Pi(\a))$ be an admissible quadruple of $(\gg, \rr)$. Then $r=T+\tau(T)$ is a symmetric solution of the ($\Pi(\rr)+\a^*$)-admissible cLYBe in the Rota-Baxter Leibniz algebra $(\gg\ltimes_{\rho^{L*}, -\rho^{L*}-\rho^{R*}}V^*,$ $\rr+\Pi(\a^*))$ of weight 0 if and only if $T$ is a weak $\mathcal{O}$-operator associated to $(V, \rho{^L}, \rho{^R})$ and $\a$.
 \item \label{it:230} Assume the validity of the $\Pi$-admissible equations, given respectively by Eqs.(\ref{eq:3})-(\ref{eq:4}), (\ref{eq:207-1})-(\ref{eq:209-1}) for $\Pi=+x$, by Eqs.(\ref{eq:3})-(\ref{eq:4}) for $\Pi=- x$, Eqs.(\ref{eq:3})-(\ref{eq:4}), (\ref{eq:215-1})-(\ref{eq:220-1}) for $\Pi\in-x+K^{\times}$ , Eqs.(\ref{eq:3})-(\ref{eq:4}), (\ref{eq:222-1})-(\ref{eq:227-1}) for $\Pi\in K^{\times}x^{-1}$, then $(V, \rho{^L}, \rho{^R},$ $ \a)$ is a representation of $(\gg, \rr)$ and there is a $(\Pi(\rr)+\a^*)$-adjoint admissible Rota-Baxter Leibniz algebra $(\gg\ltimes_{\rho^{L*}, -\rho^{L*}-\rho^{R*}}V^*,$ $\rr+\Pi(\a^*))$ of weight 0. 
    If $T$ is an $\mathcal{O}$-operator associated to $(V, \rho{^L}, \rho{^R}, \a)$, then $r=T+\tau(T)$ is a symmetric solution of the ($\rr+\Pi(\a^*))$-admissible cLYBe in the Rota-Baxter Leibniz algebra $(\gg\ltimes_{\rho^{L*}, -\rho^{L*}-\rho^{R*}}V^*,$ $\rr+\Pi(\a^*))$ of weight 0. Furthermore, there is a Rota-Baxter Leibniz algebra $(\gg\ltimes_{\rho^{L*}, -\rho^{L*}-\rho^{R*}}V^*,\rr+\Pi(\a^*)), \delta, \Pi(\rr)+\a^*)$ of weight 0, where the linear map $\delta=\delta_r$ is defined by equation(\ref{eq:94}) and $r=T+\tau(T)$.
 \end{enumerate}
 \end{cor}

 \section{Classification of 2-dimensional triangular Reynolds Leibniz bialgebras} \label{se:5} In this section, assuming the coproduct $\delta_r \neq 0$, we classify all triangular Reynolds Leibniz bialgebras. This is achieved by applying the classification of 2-dimensional Leibniz algebras over the real field $\mathbb{R}$ from \cite{AOR} and the 2-dimensional triangular Leibniz bialgebras described in \cite{LMW}. 

 \begin{thm}\label{thm:231}   Let $(\gg, [,])$ be a 2-dimensional Leibniz algebra with basis $\mathcal{B}=\{e_1,e_2\}$ , and $\lambda,\g,\eta,k_1, l_1, n_1, n_2$, be parameters. 

 \begin{enumerate}[(1)]
 \item \label{it:(1)} Let $\gg$ be a Leibniz algebra with the product given by
 \begin{center}
 \begin{tabular}{r|rr}
 $[,]$ & $e_1$  & $e_2$   \\
 \hline
 $e_1$ & $0$  & $0$ \\
 $e_2$ & $0$  & $e_1$ \\
 \end{tabular}.
 \end{center}
 According to \cite[Theorem 4.1 (a)]{LMW}, we know that $r=\eta e_1\o e_1+\gamma e_1\o e_2+\gamma e_2\o e_1$ is a symmetric solution of the cLYBe in $\gg$, and
 $\left\{
 \begin{array}{l}
 \delta_r(e_1)=0\\
 \delta_r(e_2)=\gamma e_1\o e_1\\
 \end{array}
 \right.$\hspace{-2mm}.
 Then $(\gg, [,], \d_r)$ is a triangular Leibniz bialgebra. In addition, all the Reynolds operators on $\gg$ are given by
 \begin{eqnarray*}
 &&(R1)~~\label{it:(R1)}
 \left\{
 \begin{array}{l}
 \rr(e_1)=k_1 e_1\\
 \rr(e_2)=l_1 e_1 \\
 \end{array}
 \right.\hspace{-2mm};~~(R2)~~\label{it:(R2)}
 \left\{
 \begin{array}{l}
 \rr(e_1)=k_1 e_1\\
 \rr(e_2)=l_1 e_1+\frac{2k_1}{1+\lambda k_1} e_2\\
 (k_1\neq-\frac{1}{\lambda}, \lambda\neq 0)
 \end{array}
 \right. \hspace{-2mm};~~(R3)~~\label{it:(R5)}
 \left\{
 \begin{array}{l}
 \rr(e_1)=k_1 e_1\\
 \rr(e_2)=l_1 e_1+2k_1 e_2\\
 (\lambda=0)\\
 \end{array}
 \right. \hspace{-2mm}.
 \end{eqnarray*} 
 When $\g\neq 0$, that is to say, $\d_r\neq 0$, all the triangular Reynolds Leibniz bialgebras $\big((\gg, \rr), \d_r,  S\big)$ are given as follows:
 {\small\begin{eqnarray*}
 &&(a)~~\left\{
 \begin{array}{l}
 \rr(e_1)=k_1 e_1\\
 \rr(e_2)=l_1 e_1\\
 S(e_1)=0\\
 S(e_2)=(l_1+\frac{k_1\eta}{\gamma}) e_1+k_1 e_2
 \end{array}
 \right.\hspace{-2mm};\
 ~~(b)~~\left\{
 \begin{array}{l}
 \rr(e_1)=0\\
 \rr(e_2)=l_1 e_1\\
 S(e_1)=0\\
 S(e_2)=l_1 e_1\\
 \end{array}
 \right.\hspace{-2mm};~~(c)~~\left\{
 \begin{array}{l}
 \rr(e_1)=k_1 e_1\\
 \rr(e_2)=l_1 e_1+2k_1 e_2\\
 S(e_1)=2k_1 e_1\\
 S(e_2)=(l_1-\frac{k_1\eta}{\gamma}) e_1+k_1e_2
 \end{array}
 \right.\hspace{-3mm}.
 \end{eqnarray*}
 }
 \item \label{it:(2)} Let $\gg$ be a Leibniz algebra with basis
 $\mathcal{B}=\{e_1,e_2\}$ and the product is defined by
 \begin{center}
 \begin{tabular}{r|rr}
 $[,]$ & $e_1$  & $e_2$   \\
 \hline
 $e_1$ & $0$  & $0$ \\
 $e_2$ & $e_1$  & $e_1$  \\
 \end{tabular}.
 \end{center}
 Then all the Reynolds operators on $\gg$ are given by
 \begin{eqnarray*}
 &&(R1) \left\{
 \begin{array}{l}
 \rr(e_1)=0\\
 \rr(e_2)=l_1 e_1-l_1 e_2\\
 \end{array}
 \right.\hspace{-2mm};~~(R2)\left\{
 \begin{array}{l}
 \rr(e_1)=0\\
 \rr(e_2)=l_1 e_1\\
 \end{array}
 \right.\hspace{-2mm};~~(R3)\left\{
 \begin{array}{l}
 \rr(e_1)=k_1 e_1\\
 \rr(e_2)=(k_1-\frac{1}{\lambda}) e_1+\frac{1}{\lambda} e_2\\
 (\lambda \neq0)
 \end{array}
 \right.\hspace{-2mm}.
 \end{eqnarray*}
 \begin{enumerate}[(I)]
 \item \label{it:(I)} According to \cite[Theorem 4.1 (b)]{LMW}, we know that $r=\eta e_1\o e_1+\gamma e_1\o e_2+\gamma e_2\o e_1$ is the symmetric solution of the cLYBe in $\gg$, and
 $$\left\{
 \begin{array}{l}
 \delta_r(e_1)=0\\
 \delta_r(e_2)=(\eta+\gamma) e_1\o e_1+\gamma e_1\o e_2\\
 \end{array}
 \right.\hspace{-2mm},$$
 then $(\gg, [,], \d_r)$ is a triangular Leibniz bialgebra.
 \item   \label{it:(II)} According to \cite[Theorem 4.1 (b)]{LMW}, we know that $r=\eta e_1\o e_1-\eta  e_1\o e_2-\eta  e_2\o e_1+\eta  e_2\o e_2$ is the symmetric solution of the cLYBe in $\gg$, and
 $$\left\{
 \begin{array}{l}
 \delta_r(e_1)=\eta(e_1\o e_2-e_2\o e_1)\\
 \delta_r(e_2)=\eta(e_1\o e_2-e_2\o e_1)\\
 \end{array}
 \right.\hspace{-2mm}.$$
 \end{enumerate}
 Then $(\gg, [,], \d_r)$ is a triangular Leibniz bialgebra.

 For Case \ref{it:(I)}, if $\d_r\neq 0$, then all the triangular Reynolds Leibniz bialgebras $\big((\gg, \rr), \d_r, S\big)$ are given as follows: 
 \begin{eqnarray*}
 &(a)\left\{
 \begin{array}{l}
 \rr(e_1)=0\\
 \rr(e_2)=0\\
 S(e_1)=0\\
 S(e_2)=n_1e_1+n_2e_2\\
 \end{array}
 \right.\hspace{-2mm},\
 ~~ (b)\left\{
 \begin{array}{l}
 \rr(e_1)=0\\
 \rr(e_2)=l_1 e_1-l_1e_2\\
 S(e_1)=0\\
 S(e_2)=n_1e_1-n_1e_2\\
 \end{array}
 \right.\hspace{-2mm},\
 ~~(c)\left\{
 \begin{array}{l}
 \rr(e_1)=0\\
 \rr(e_2)=-\frac{1}{\lambda}e_1+\frac{1}{\lambda}e_2\\
 S(e_1)=\frac{1}{\lambda}e_1\\
 S(e_2)=\frac{1}{\lambda}e_1\\
 (\l\neq 0)
 \end{array}
 \right.\hspace{-2mm},&\\
 &(d)\left\{
 \begin{array}{l}
 \rr(e_1)=0\\
 \rr(e_2)=l_1 e_1\\
 S(e_1)=0\\
 S(e_2)=n_1e_1\\
 \end{array}
 \right.\hspace{-2mm},~~~(e)\left\{
 \begin{array}{l}
 \rr(e_1)=0\\
 \rr(e_2)=l_1 e_1\\
 S(e_1)=0\\
 S(e_2)=l_1e_1\\
 \end{array}
 \right.\hspace{-2mm},~~(f)\left\{
 \begin{array}{l}
 \rr(e_1)=0\\
 \rr(e_2)=l_1 e_1\\
 S(e_1)=0\\
 S(e_2)=n_1e_1+n_2e_2\\
 \end{array}
 \right.\hspace{-2mm},&\\
 &(g)\left\{
 \begin{array}{l}
 \rr(e_1)=\frac{2n_2}{1+\lambda n_2} e_1\\
 \rr(e_2)=(\frac{2n_2}{1+\lambda n_2}-\frac{1}{\lambda}) e_1+\frac{1}{\lambda} e_2\\
 S(e_1)=\frac{2n_2}{1+\lambda n_2}e_1\\
 S(e_2)=(\frac{2n_2}{1+\lambda n_2}-n_2)e_1+n_2e_2\\
 (\lambda \neq 0, 1+\lambda n_2\neq 0 )\\
 \end{array}
 \right.\hspace{-2mm},~~~(h)\left\{
 \begin{array}{l}
 \rr(e_1)=\frac{2}{\lambda} e_1\\
 \rr(e_2)=\frac{1}{\lambda} e_1+\frac{1}{\lambda} e_2\\
 S(e_1)=\frac{1}{\lambda}e_1\\
 S(e_2)=-\frac{1}{\lambda}e_1+\frac{2}{\lambda}e_2\\
 (\lambda\neq 0)
 \end{array}
 \right.\hspace{-2mm},&\\
 &(i)\left\{
 \begin{array}{l}
 \rr(e_1)=0\\
 \rr(e_2)=-\frac{1}{\lambda}e_1+\frac{1}{\lambda} e_2\\
 S(e_1)=0\\
 S(e_2)=0\\
 (\lambda\neq 0)\\
 \end{array}
 \right.\hspace{-2mm},~~~(j)\left\{
 \begin{array}{l}
 \rr(e_1)=0\\
 \rr(e_2)=-\frac{1}{\lambda}e_1+\frac{1}{\lambda} e_2\\
 S(e_1)=\frac{1}{\lambda}e_1\\
 S(e_2)=\frac{1}{\lambda}e_1\\
 (\lambda\neq 0)\\
 \end{array}
 \right.\hspace{-2mm}.&
 \end{eqnarray*}

 For Case \ref{it:(II)},  if $\d_r\neq 0$, then all the triangular Reynolds Leibniz bialgebras $\big((\gg, \rr), \d_r, S\big)$ are given as follows:
 \begin{eqnarray*}
 && (a)\left\{
 \begin{array}{l}
 S(e_1)=0\\
 S(e_2)=l_1 e_1-l_1e_2\\
 \rr(e_1)=0\\
 \rr(e_2)=l_1 e_1-l_1e_2\\
 \end{array}
 \right.\hspace{-2mm}.
 \end{eqnarray*}
\end{enumerate}

 \end{thm}

\bigskip

\section*{Acknowledgments}
This work is supported by National Natural Science Foundation of China (No. 12471033).
\smallskip
\smallskip

\noindent
{\bf Declaration of interests.} The authors have no conflicts of interest to disclose.

\noindent
{\bf Data availability.} Data sharing is not applicable as no new data were created or analyzed.

\end{document}